\documentclass[12pt]{article}

\usepackage[margin=1in]{geometry}
\usepackage[utf8]{inputenc}
\usepackage{listings}
\usepackage{graphicx,epstopdf}

\usepackage{pgfplots}
\usepackage{tikz}
\usepackage{subfig}
\usepackage{caption}

\usepackage[normalem]{ulem}
\usepackage{ftnxtra}

\usepackage{rotating}
\usepackage{graphicx}
\usepackage{adjustbox}
\usepackage{amsmath,amssymb,amsthm,xcolor,graphicx,verbatim,caption,enumerate,algorithm,algorithmic}

\usepackage{cleveref}
\usepackage{float}
\usepackage{multirow}
\usepackage{booktabs}
\newtheorem{theorem}{Theorem}[section]

\newtheorem{remark}[theorem]{Remark}
\newtheorem{lemma}[theorem]{Lemma}
\newtheorem{proposition}[theorem]{Proposition}
\newtheorem{ex}[theorem]{Example}

\newtheorem{definition}[theorem]{Definition}

\newcommand{\norm}  [1]{\ensuremath{\left  \|       #1  \right \|       }}

\newcommand{\rev}[1]{\textcolor{black}{#1}}
\newcommand{\revv}[1]{\textcolor{black}{#1}}
\newcommand{\revvv}[1]{\textcolor{black}{#1}}

\newcommand{\cl}  {{\rm cl  \,}}
\newcommand{\bd}  {{\rm bd \,}}

\newcommand{\Int} {{\rm int \,}}
\newcommand{\conv}  {{\rm conv \,}}
\newcommand{\cone}{{\rm co\,}}

\newcommand{\Vused}{V_{\textnormal{used}}\,} 
\newcommand{\Vinfo}{V^k_{\textnormal{info}}\,} 

\newcommand{\X}{\mathcal{X}}
\newcommand{\R}{\mathbb{R}}

\renewcommand{\P}{\mathcal P}
\newcommand{\T}{\mathsf T}

\newcommand{\Minc}{{\rm Min}_C\,}
\newcommand{\wMinc}{{\rm wMin}_C\,}

\newcommand{\recc}{{\rm rec \,}}
\DeclareMathOperator*{\argmin}{arg\,min}
\DeclareMathOperator*{\argmax}{arg\,max}

\makeatletter
\newenvironment{procedure}[1][htb]{%
	\renewcommand{\ALG@name}{Procedure}
	\begin{algorithm}[#1]%
	}{\end{algorithm}}
\makeatother

\author{\.Irem Nur Keskin \thanks{Fuqua School of Business, Duke University, Durham, NC 27708, USA, iremnur.keskin@duke.edu} 
	\and Firdevs Ulus \thanks{Bilkent University, Department of Industrial Engineering, Ankara, 06800 Turkey, firdevs@bilkent.edu.tr}}

\title{Outer approximation algorithms for convex vector optimization problems}

\date{\today}

\begin{document}
	\maketitle
	
	\begin{abstract} \noindent	
		In this study, we present a general framework of outer approximation algorithms to solve convex vector optimization problems, in which the Pascoletti-Serafini (PS) scalarization is solved iteratively. This scalarization finds the minimum `distance’ from a reference point, which is usually taken as a vertex of the current outer approximation, to the upper image through a given direction. We propose efficient methods to select the parameters (the reference point and direction vector) of the PS scalarization and analyze the effects of these on the overall performance of the algorithm. Different from the existing vertex selection rules from the literature, the proposed methods do not require solving additional single-objective optimization problems.	Using some test problems, we conduct an extensive computational study where three different measures are set as the stopping criteria: the approximation error, the runtime, and the cardinality of \revv{the} solution set. We observe that the proposed variants have satisfactory results\revv{,} especially in terms of runtime compared to the existing variants from the literature.		
		\medskip
		
		
		\medskip
		
	\end{abstract}

	\section{Introduction} \label{ch:intro}
	
		Multiobjective optimization refers to optimizing \rev{multiple} conflicting objectives simultaneously \rev{and it is a useful tool} for many \rev{applications} in various fields from finance to engineering \rev{since, by the nature of applications,} \revv{there may be a trade-off among several objectives}. For a \rev{multiobjective optimization problem (MOP)}, there is no single solution optimizing all objectives. Instead, feasible solutions that cannot be improved in one objective without deteriorating in at least another objective, namely \emph{efficient solutions}, are of interest. \rev{In many applications, instead of finding the set of all efficient solutions, it is sufficient to obtain the images of efficient solutions, namely the} \emph{nondominated points}, in the objective space. 
	
	An optimization problem that requires \revv{minimizing or maximizing} a vector-valued objective function with respect to a partial order induced by an ordering cone $C$ is referred to as \revv{a} vector optimization problem (VOP) \rev{and it is also widely used in different fields: for instance, see \cite{Hamel,Rudloff2021} for applications in financial mathematics.} MOPs can be seen as special instances of VOPs where the ordering cone is the nonnegative orthant. The terminology \rev{in vector optimization is slightly different from the multiobjective optimization terminology. In particular, for a minimization problem where the ordering cone is $C$, a \emph{minimizer} is a feasible solution that cannot be improved with respect to the order relation induced by $C$. The image of a minimizer is then a \emph{$C$-minimal point} in the objective space.} 
	
	
	One of the most common approaches to \rev{obtain} a minimizer for a VOP is to solve \rev{a scalarization problem}, which is a single objective optimization problem \rev{induced by the VOP}. In general, a scalarization model is parametric and has the potential to generate a `representative' set of minimizers when solved for different set\revv{s} of parameters. Throughout, two scalarization methods will be used. The first one is the well-known weighted sum scalarization \cite{gass1955}, which is performed by optimizing the weighted sum of the objectives over the original feasible region. In addition, we use the Pascoletti-Serafini (PS) scalarization \cite{pascolettiSerafini1984}, which aims to find the closest $C$-minimal point from a given reference point through a given direction parameter. Unlike weighted sum scalarization, it has \revv{the} potential to find all minimizers of a given VOP even if the problem is not convex.
	
	In this paper, we focus on convex vector optimization problems (CVOPs). In the literature,  there are iterative algorithms utilizing scalarization models to \rev{generate an approximation of} the set of all $C$-minimal points in the objective space. To the best of our knowledge, the first such algorithm is proposed in 1998, by Benson, \cite{benson} and it \revv{is} designed to solve linear MOPs. It generates the set of all nondominated points \rev{in the objective space} by iteratively obtaining improved polyhedral outer approximations of it. Later, this algorithm is generalized to solve linear VOPs; moreover, using geometric duality results, a geometric dual counterpart of this algorithm is also proposed, see \cite{lohne}. In 2011, Ehrgott et al. \cite{ehrgottSS2011} extended the linear VOP algorithm from \cite{lohne} to solve CVOPs. \revvv{Then, in 2014, L\"ohne et al. \cite{cvop2014} proposed a similar algorithm 
		and a geometric dual variant.}
	
	\revvv{On the other hand, in 2003, Klamroth et al. \cite{Klamroth2002} proposed approximation algorithms for convex and non-convex MOPs. The mechanism of the outer approximation algorithm for the convex case is similar to Benson's algorithm \cite{benson}. 
		Different from \cite{benson}, in \cite{Klamroth2002} the selection for the reference point for the scalarization is not arbitrary. 
		In \cite{Klamroth2002}, there is also an inner approximation algorithm for the convex problems and 
		the convergence rates of both algorithms are provided for the biobjective case.
	}
	
	\revvv{Recently, Dörfler et al. \cite{Dorfler2020} proposed a variant {of} Benson's algorithm for CVOPs that includes a vertex selection procedure {that} also yields a direction parameter for the PS scalarization. 
		To compute the parameters, 
		the algorithm solves quadratic programming problems for the vertices of the current outer approximation.}
	
	\revvv{In addition to the algorithms which solve PS scalarization (or equivalent models), recently Ararat et al. \cite{umer2020} proposed an outer approximation algorithm {that} solves norm-minimizing scalarizations. This scalarization does not require a direction parameter but only a reference point, which is again selected among the vertices of the current outer approximation. }

	The aforementioned algorithms differ in terms of the selection procedures for the parameters of the scalarization. The ones from \cite{Klamroth2002,Dorfler2020} require solving additional models to select a vertex at each iteration. This feature enables them to provide the current approximation error at each iteration of the algorithm at the cost of solving \revv{a} considerable number of models. The rest of the algorithms do not solve additional models and they provide the approximation error after termination. For the selection of the direction parameter, \cite{Klamroth2002, ehrgottSS2011} use a fixed point from the upper image; \cite{Dorfler2020} uses the point from the inner approximation that yield\revv{s} the minimum distance to the selected vertex; and \cite{cvop2014} uses a fixed direction from the interior of the ordering cone through the algorithm. \Cref{table:alg} summarizes these properties.
	
		\begin{table}[h] 
		\centering
		\resizebox{\textwidth}{!}{
			\begin{tabular}{cccccc}
				\textbf{Algorithm} & \textbf{\begin{tabular}[c]{@{}c@{}}Finiteness / \\ Convergence\end{tabular}} & \textbf{\begin{tabular}[c]{@{}c@{}}Choice of \\ Direction\end{tabular}} & \textbf{\begin{tabular}[c]{@{}c@{}}Vertex\\ Selection (VS)\end{tabular}}   & \textbf{\begin{tabular}[c]{@{}c@{}}Models \\ Solved for VS\end{tabular}} & \textbf{\begin{tabular}[c]{@{}c@{}}Approximation \\ Error\end{tabular}} \\ \hline \hline
				Klamroth et al. \cite{Klamroth2002}   & \begin{tabular}[c]{@{}c@{}}Convergence \\ for biobjective           \end{tabular}                                                           & \begin{tabular}[c]{@{}c@{}}Inner point\\ (fixed)\end{tabular}        & \begin{tabular}[c]{@{}c@{}}Distance \\ to upper image\end{tabular}     & \begin{tabular}[c]{@{}c@{}}Gauge-based \\ Model\end{tabular}                     & \begin{tabular}[c]{@{}c@{}}At each \\ iteration\end{tabular}                      \\ \hline
				Ehrgott et al. \cite{ehrgottSS2011}    & -                                                                      & \begin{tabular}[c]{@{}c@{}}Inner point\\ (fixed)\end{tabular}           & Arbitrary                                                                  & -                                                                                & \begin{tabular}[c]{@{}c@{}}After \\ Termination\end{tabular}                                                                         \\ \hline
				Löhne et al. \cite{cvop2014}      & -                                                                      & Fixed                                                                   & Arbitrary                                                                  & -                                                                                & \begin{tabular}[c]{@{}c@{}}After \\ Termination\end{tabular}                                                                         \\ \hline
				Dörfler et al. \cite{Dorfler2020}     & -                                                                      & \begin{tabular}[c]{@{}c@{}}Inner point\\ (changing)\end{tabular}        & \begin{tabular}[c]{@{}c@{}}Distance to inner \\ approximation\end{tabular} & \begin{tabular}[c]{@{}c@{}}Quadratic \\ Model\end{tabular}                       & \begin{tabular}[c]{@{}c@{}}At each \\ iteration\end{tabular}                      \\ \hline
				Ararat et al. \cite{umer2020}     & Finiteness                                                                    & \begin{tabular}[c]{@{}c@{}}Not \\ Relevant\end{tabular}        & \begin{tabular}[c]{@{}c@{}}Arbitrary\end{tabular} & \begin{tabular}[c]{@{}c@{}} - \end{tabular}                       & \begin{tabular}[c]{@{}c@{}}After \\ Termination\end{tabular}  
				\\ \hline		\end{tabular}}
		\caption{Existing outer approximation algorithms to solve CVOPs}
		\label{table:alg}
	\end{table}

\rev{The main contributions of this paper can be listed as follows: (1) We} present a general framework of the outer approximation algorithms from the literature\rev{; (2) propose various additional variants, which \revv{select the two parameters of the PS scalarizations in structured and efficient ways}; and (3) compare the performances through numerical tests. In particular,} after proposing different direction selection rules and conducting a preliminary computational study, we propose three vertex selection rules. \revv{Different from the vertex selection rules from the literature, these rules are not based on solving additional single-objective optimization problems, hence \revvv{they are} computationally more efficient.} The first one benefits from the clustering of the vertices. The second rule selects the vertices using the adjacency information among them. For the last one, we employ a procedure \revv{that} creates local upper bounds to the nondominated points. This procedure is first introduced in the context of multiobjective combinatorial optimization, see for instance \cite{Klamroth2}, and by its design, it works only for convex MOPs. Together with the proposed variants, we also implement some of the algorithms from the literature and we provide an extensive computational study to observe the\revv{ir} behavior under different stopping conditions. 

In Sections \ref{ch:Prelim} and \ref{ch:problem}, we provide the preliminaries and the solution concepts together with the scalarization models used in the paper, respectively. In \Cref{ch:algorithm_variants}, we present the general framework of the outer approximation algorithm; \revv{the} construction of different direction and vertex selection rules; and similar algorithms from the literature. \revvv{We provide the test instances and some preliminary computational results in \Cref{sect:Comput_pre}.} \Cref{ch:comp_results} includes \revvv{the main} computational study where we compare the proposed variants with the algorithms from the literature. Finally, we conclude the paper in \Cref{ch:conclusion}. 
	
\section{Preliminaries}
\label{ch:Prelim}
\rev{In this section, we provide the \revv{notation} and definitions that are used throughout.} 

\rev{For $p\geq 2$,  $\R^p$ is the $p$ dimensional Euclidean space and $\norm{\cdot}$ is the Euclidean norm on $\R^p$. The closed ball around a point $a\in\R^p$ with radius $r> 0$ is denoted by $B(a,r):=\{y \in\R^p \mid \norm{y-a}\leq r\}.$ Moreover, we use the following notation: $e:=(1,\ldots,1)^\T \in \R^p$, $e^j\in \R^p$ is the unit vector with $j^{th}$ component being 1 and	$\R^p_+:= \{y \in \R^p \mid y\geq 0\}$.}

Let $S$ be a subset of $\R^p$. The convex hull, conic hull, interior, closure, and boundary of  $S$ \revv{are} denoted by $\conv S$, $\cone S$, $\Int S$, $\cl S$, and $\bd S$ respectively. \rev{A vector} $z\in\mathbb{R}^{p}\setminus \{0\}$ is a \textit{recession direction} of $S$, if $y+ \gamma z\in S$ for all \revv{$\gamma \geq 0, y\in S$}. The set of all recession directions of $S$ is the \emph{recession cone of $S$} and \revv{is} denoted by $\recc S$. 

\rev{Let $S \subseteq \R^p$ further be a convex set. A hyperplane given by $\{y \in \R^p \mid a^\T y = b\}$ for some $a\in\R^p\setminus\{0\}, b\in \R$ is a \emph{supporting hyperplane} of \revv{$S$ if} $S \subseteq \{y\in \R^p \mid a^\T y \geq b\}$ and there exists $s\in S$ with $a^\T s = b$. A convex subset $F\subseteq S$} 
is \rev{called} a \textit{face} of $S$ if $\lambda y^1+(1-\lambda) y^2\in F$ \rev{with $y^1,y^2 \in S$ and} $0 < \lambda <1$ imply \rev{$y^1,y^2 \in F$}. A zero-dimensional face is an \emph{extreme point} (or \emph{vertex}) \rev{and} a one-dimensional face is an \emph{edge} of $S$. A recession direction $z \in \R^p \setminus \{0\}$ of convex set $S$ is said to be an \textit{extreme direction} of $S$ if $\{v+rz \in \R^p \mid r \geq 0\}$ is a face for some extreme point $v$ of $S$, see~\cite[Section 18]{rockafellar}. 

\rev{The set} $S$ is a \textit{polyhedral convex set} \revv{if} it is of the form $S =\{y\in \mathbb{R}^{p}\ | \ A^Ty\geq b\}$, where $A\in \mathbb{R}^{p\times k}$, $b\in \mathbb{R}^{k}$. \rev{This form is called a halfspace representation of $S$.} If $S$ has at least one \rev{vertex, then} it can also be represented as $S=\conv \rev{V^S} + \cone \conv \rev{D^S}$, where $V^S \subseteq\R^p$ \rev{and $D^S \subseteq\R^p$ are} the finite sets of vertices \rev{and extreme directions of $S$, respectively. For a polyhedral convex set $S$, we call two vertices \emph{adjacent} if the line segment between them is an edge of $S$.}

The problem of finding \rev{the set of all vertices and extreme directions of a polyhedral convex set $S$, given its halfspace representation} 
is called the \emph{vertex enumeration problem}. Throughout, we employ the \emph{bensolve tools} for \revvv{this purpose}~\cite{bensolve,lohneWeissing2017}. In addition to the vertices $V$ and directions $D$ \rev{of $S$}, \emph{bensolve tools} also yields the set of all adjacent vertices $V^v_{adj}\subseteq V$ of each vertex $v \in V$. 

Let $T \subseteq \R^p$ \rev{be another subset of $\R^p$}. The Minkowski sum of $S$ and $T$ \rev{is} $S+T:=\{s+t \in \R^p \mid s \in S, t \in T\}$. Moreover, for $\alpha \in \R$, we have $\alpha S:=\{\alpha s \in \R^p \mid s \in S\}$ \rev{and we} denote the set $S+(-1)\cdot T $ \rev{simply by $S-T:=\{s-t \in \R^p \mid s \in S, t \in T\}$.} 
\revv{\begin{definition}
		Let $S,T\subseteq \R^p$. The Hausdorff distance between $S$ and $T$ is defined as
		$$d_H(S,T) := \max\{\sup_{s \in S} d(s,T),\sup_{t \in T }  d(t,S) \},$$  where $d(t,S) := \inf_{s\in S}\norm{t-s}$ is the distance from point $t$ to set $S$. 
\end{definition}}
\rev{The following result provides a simple way of measuring the Hausdorff distance between a polyhedral set and a convex subset of it. The proof is omitted since it follows the same steps as the proof of \cite[Lemma 5.3]{umer2020}.
	\begin{lemma} \label{lem:Hausdorffdist}
		Let $S$ and $T$ be convex sets in $\R^p$ with $\recc S= \recc T$ and $T \subseteq S$. If $S$ is polyhedral convex with at least one vertex, then $d_H(S,T)=\max_{v \in V^S} d(v,T)$ where $V^S$ is the set of all vertices of $S$.
	\end{lemma} 
}	

Let $C \subseteq \R^p$ be a \revv{nonempty} closed convex cone. $C$ is said to be \emph{pointed} if $C \cap (-C) = \{0\}$ 
and \emph{non-trivial} if $C \neq \{0\}, C\neq \R^p$. If $C$ is pointed \revv{and 
	non-trivial, 
	then a partial order} $\leq_C$ is defined as follows: \revv{for} all $y^1,y^2\in \R^p$, $y^1\leq_C y^2$ holds if and only if $y^2-y^1\in C$.

\rev{The \emph{dual cone} of $C$ is defined as $C^+:=\{a \in \R^p \mid \revv{\forall x \in C: a^\T x\geq 0} \}$ and it is a closed convex cone. If $C\subseteq \R^p$ is a polyhedral cone, then $C^+$ is also polyhedral and can be written as $C^+ = \cone \conv\{w^1,\ldots,w^l\}$, where $w^1,\ldots,w^l\in \R^p$ are the extreme directions of $C^+$ for some $l\geq 0$. In this case, $y^1\leq_C y^2$ holds if and only if $(w^i)^\T y^1 \leq (w^i)^\T y^2$ for all $i\in \{1,\ldots, l\}$.}

\revv{Let $C \subseteq \R^p$ be a nonempty convex pointed cone.} \rev{The following definitions are based on the partial order $\leq_C$ and they are well-known concepts in convex vector optimization.
	\begin{definition} \cite[Ch. 2, Definition 2.1]{luc} \revvv{For a set $S\subseteq \R^p$, a} point $s \in S$ is a \emph{$C$-minimal element of $S$} if $\left(\{s\}-C\setminus\{0\}\right)\cap S= \emptyset$\revv{. If $C$ has nonempty interior and $(\{s\}-\Int C)\cap S = \emptyset$, then $s \in S$ is called} a \emph{weakly $C$-minimal element of $S$}. The set of all $C$-minimal (weakly $C$-minimal) elements of $S$ is denoted by $\Minc S$ ($\wMinc S$). 
	\end{definition}
	\begin{definition}	\cite[Ch. 1, Definition 6.1]{luc}
		A function $f\colon \R^n \rightarrow \R^p$ is said to be \emph{$C$-convex} if $f(\lambda x + (1-\lambda) y) \leq_C \lambda f(x) +(1-\lambda) f(y)$ for all \revv{$x,y \in \R^n, \lambda \in [0,1]$.}
\end{definition}}
	
	\section{The problem}
	\label{ch:problem}
	We consider a convex vector optimization problem given by
	\begin{align*}\label{P}
		\textrm{minimize } &{f(x)} \textrm{ with respect to } {\leq_C}  \textrm{ subject to } x \in \X, \tag{P} &
	\end{align*}
	where $C$ is a closed convex pointed cone with nonempty interior, $\X \subseteq \R^n$ is a nonempty closed convex set and $f: \R^n \rightarrow \R^p$ is a continuous $C$-convex function given by $f(x):=(f_1(x),\ldots,f_p(x))^\T$. Throughout, we assume that $C$ is polyhedral, \rev{in particular, its dual cone is given by $C^+:=\cone \conv\{w^1,\ldots,w^l\}$,} and $\X$ is compact with nonempty interior. $f(\X):= \{f(x) \in \R^p \mid x \in \X \}$ is the image of $\X$ under $f$ and the \emph{upper image} for problem $\eqref{P}$ is defined as $\mathcal{P}:= \cl(f(\X)+C)$, \rev{see Figure \ref{fig:defn}}. $\mathcal{P}$ is a closed convex set and under the assumptions of the problem, $\mathcal{P} = f(\X)+C$ \cite{umer2020}. Moreover, the set of all weakly $C$-minimal points of $\mathcal{P}$ is $\bd \mathcal{P}$ \cite{ehrgottSS2011}.
	
	\rev{Different from a single objective optimization problem, there are various optimality and solution concepts for \revv{the} problem \eqref{P}. We will now introduce the solution concepts that will be used throughout.} A feasible point $\bar{x} \in \mathcal{X}$ of \eqref{P} is called \emph{(weak) minimizer}, if $f(x)$ is (weakly) $C$-minimal element of $f(\X)$. In the context of multiobjective optimization where the ordering cone is $C=\R^p_+$, (weak) minimizers are referred to as (weak) efficient solutions. Another important concept for the multiobjective case \rev{that will be used through the rest of the paper} is the \emph{ideal point}, which is denoted by $y^I\in\R^p$ and defined as $y^I_i:=\inf\{f_i(x) \mid x \in \mathcal{X}\}$ for $i=1,\ldots,p$. 
	
	Problem \eqref{P} is said to be \emph{bounded} if $\mathcal{P} \subseteq \{y\} + C$ for some $y \in \R^p$. \rev{Since $\mathcal{X}$ is assumed to be compact, we know that} problem \eqref{P} is bounded \cite{cvop2014}. Below, we provide two solution concepts for bounded convex vector optimization problems, which are motivated from a set-optimization point of view \cite{lohne}. The first solution concept depends on a fixed direction parameter $c \in \Int C$. It is introduced in \cite{cvop2014} as \revv{a} `finite $\epsilon$-solution'. Here we add the term `with respect to $c$' to emphasize this dependence. The second one is free of a direction parameter but depends on a norm in $\R^p$. 
	
	\begin{definition}[{\cite[Definition 3.3]{cvop2014}}]
		\label{defn:weakepssol_c}
		Let $c\in \Int C$ be fixed. \revv{For $\epsilon > 0$,} a nonempty finite set $\bar{\X} \subseteq \X$ of (weak) minimizers is a \emph{finite (weak) $\epsilon$-solution with respect to $c$} if $\conv f(\mathcal{\bar{X}})+C-\epsilon \{c\} \supseteq \mathcal{P}.$
	\end{definition}
	
	\begin{definition}[\cite{Dorfler2020,umer2020}]
		\label{defn:weakepssol}
		\revv{For $\epsilon > 0$,} a nonempty finite set $\bar{\X} \subseteq \X$ of (weak) minimizers is a \emph{finite (weak) $\epsilon$-solution} if $\conv f(\mathcal{\bar{X}})+C+B(0,\epsilon) \supseteq \mathcal{P}.$
	\end{definition}
	
	\rev{Illustrations of an upper image and Definitions \ref{defn:weakepssol_c} and \ref{defn:weakepssol} can be seen in Figure \ref{fig:defn}.}

	
	\begin{figure}[!]
	\centering
	\caption{\rev{Illustrations of sets $f(\X)$, $\P$ (left) and solution concepts given by \Cref{defn:weakepssol_c} (middle) and \Cref{defn:weakepssol} (right). The dashed lines show $\bd(\conv f(\mathcal{\bar{X}})+C)$ (middle and right) and solid blue lines show $\bd(\conv f(\mathcal{\bar{X}})+C-\epsilon \{c\})$ (middle) and $\bd(\conv f(\mathcal{\bar{X}})+C+B(0,\epsilon))$ (right). }}
	\begin{minipage}[b]{0.28\linewidth}
		\centering
		\includegraphics[width=\textwidth]{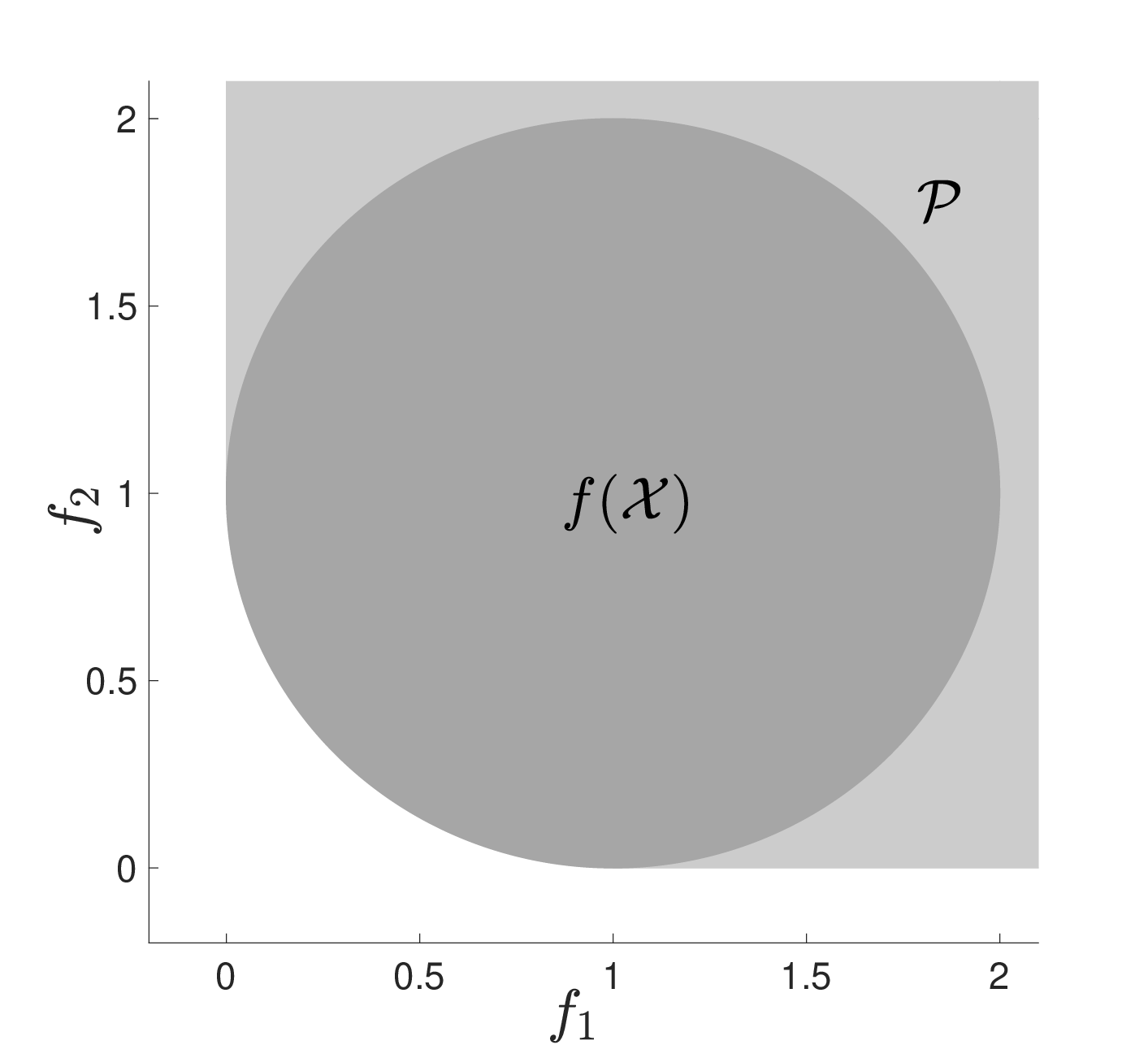}
	\end{minipage}
	\begin{minipage}[b]{0.3\linewidth}
		\centering
		\includegraphics[width=\textwidth]{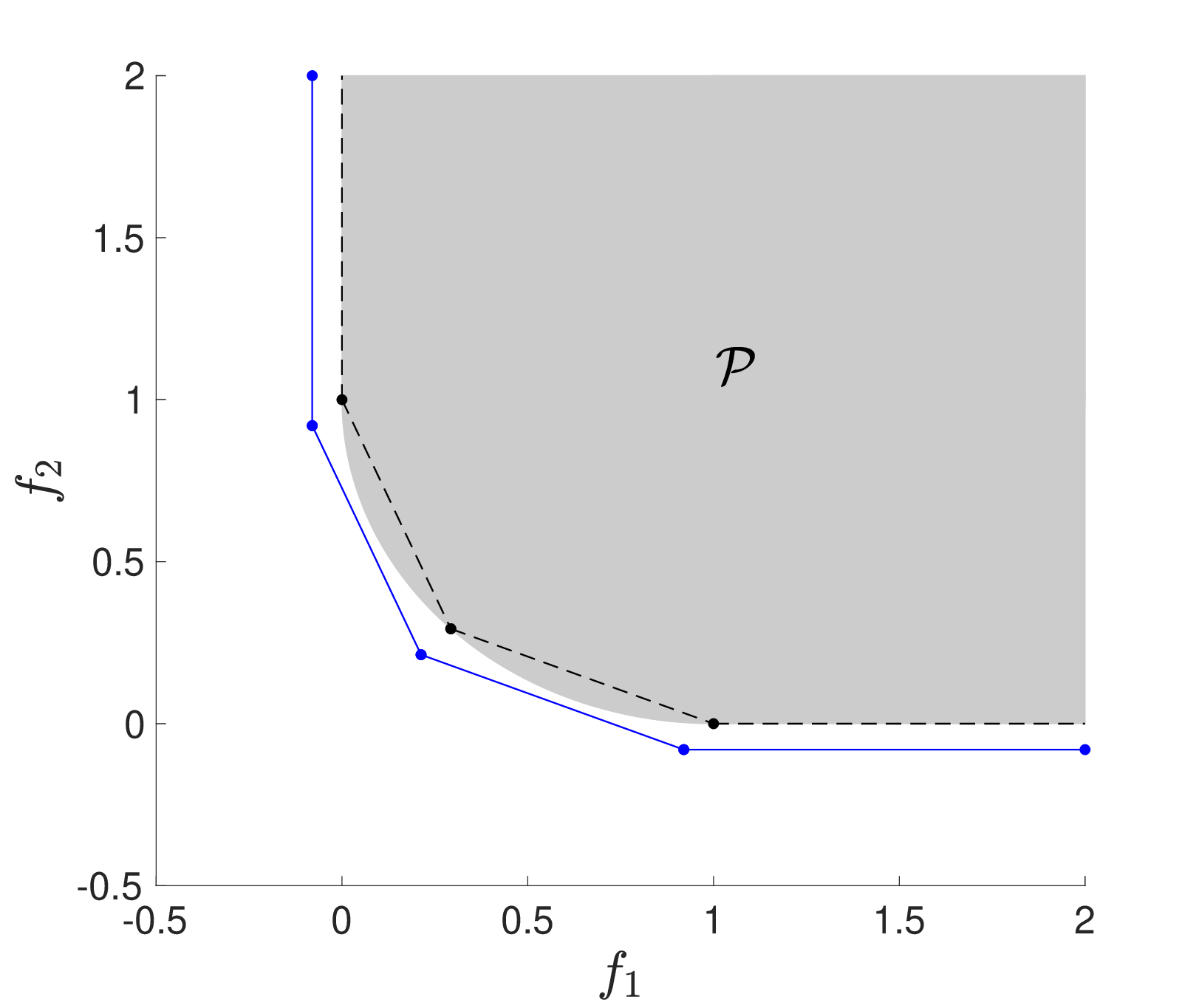}
	\end{minipage}
	\begin{minipage}[b]{0.3\linewidth}
		\centering
		\includegraphics[width=\textwidth]{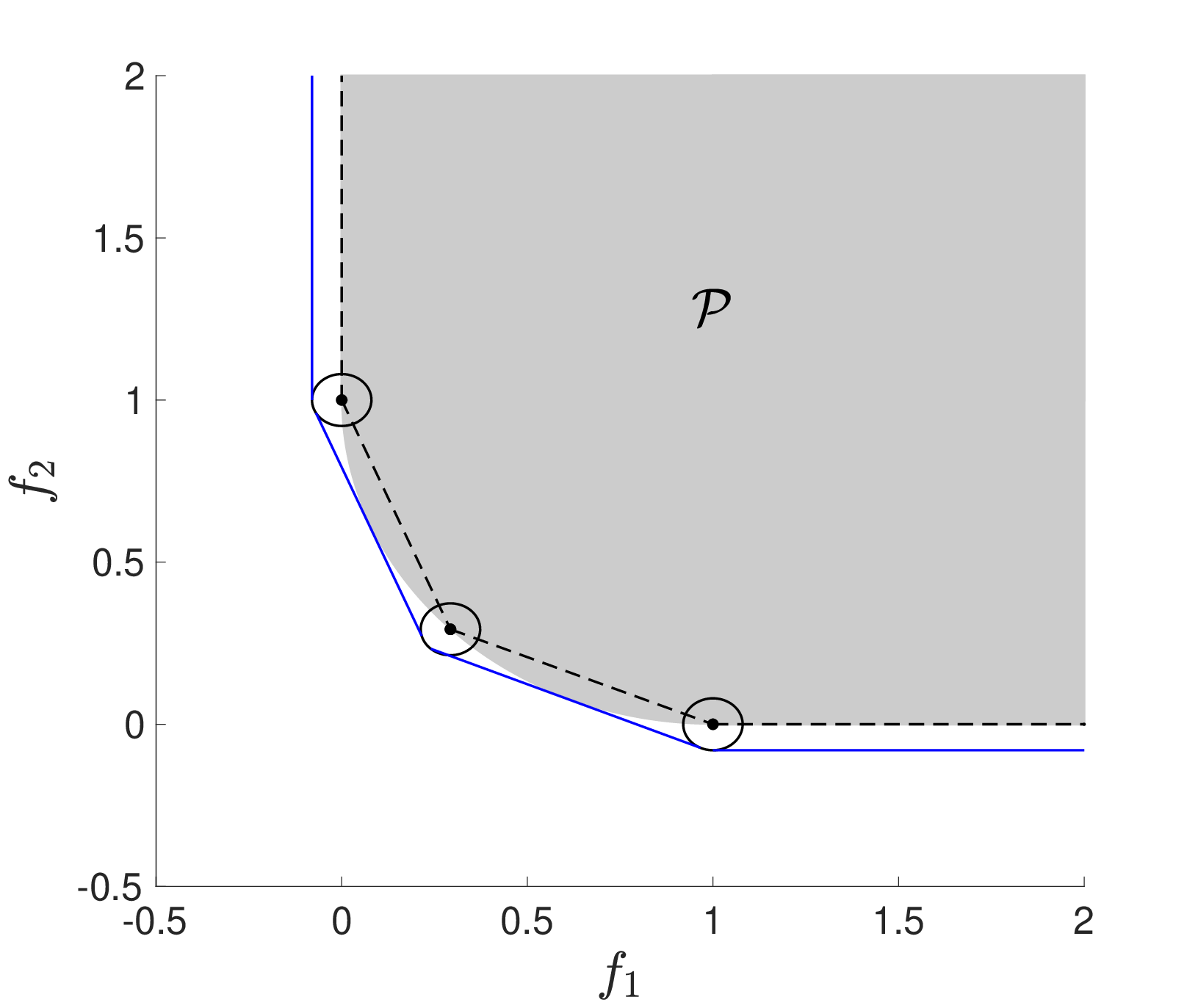}
	\end{minipage}
	\label{fig:defn}
\end{figure}
	
	Note that both of these solution concepts yield an outer approximation to the upper image. On the other hand, $\conv f(\bar{\X})+C$ is an inner approximation to $\mathcal{P}$. The Hausdorff distance between these inner and outer sets \revv{is} bounded \cite{Dorfler2020}. Moreover, as the feasible region $\X$ is compact, for any $\epsilon > 0$, there exists a finite weak $\epsilon$-solution (with respect to $c$) to problem~\eqref{P}, see~\cite[Proposition 4.3]{cvop2014} and~\cite[Proposition 3.8]{umer2020}.
	
	\rev{There are different solution approaches to solve bounded CVOPs in the sense of Definitions \ref{defn:weakepssol_c} or \ref{defn:weakepssol}. The main idea of these approaches is to generate (weak) minimizers for \eqref{P} in a structured way.} One way of generating (weak) minimizers is to solve scalarization models. Below, we provide two well-known scalarization models together with some results regarding them.
	
	For a weight parameter $w\in\R^p$, the weighted sum scalarization model is given by
	\begin{align*}  \label{WS}
		\tag{WS($w$)}
		\textrm{minimize }  {w^\T f(x)} \:\: \textrm{subject to } x \in \mathcal{X}.
	\end{align*}
	
	\begin{proposition}[{\cite[\revv{Corollary} 5.29]{jahn}}]\label{prop:ws}
		An optimal solution $x \in \X$ of \eqref{WS} is a weak minimizer of~\eqref{P} if $w \in C^+ \setminus \{0\}$. Conversely, for any weak minimizer $x\in\X$, there exists $w \in C^+ \setminus \{0\}$ such that $x$ is an optimal solution of~\eqref{WS}.
	\end{proposition}
	
	
	
	Pascoletti-Serafini scalarization \cite{pascolettiSerafini1984} is given by
	\begin{align*} \label{PS}
		\tag{PS($v,d$)}
		\textrm{minimize } z \:\: \textrm{subject to } f(x) \leq_C v + zd, \: x \in \X, \: z \in \R.
	\end{align*}
	The parameters $v, d \in \R^p$ are referred to as the reference point and the direction, respectively. \rev{The Lagrange dual of the PS problem can be written as 
		\begin{align*} \label{PS-dual}
		\tag{D-PS($v,d$)}
		\textrm{maximize } \inf_{x \in \mathcal{X}} w^\T f(x) - w^\T v \:\: \textrm{subject to } w^\T d = 1, \: w \in C^+,
	\end{align*}
see \cite{Dorfler2020,cvop2014}.} Below, we provide some well-known results regarding \eqref{PS} \rev{and \eqref{PS-dual}}.
	
	\begin{proposition} \cite[Proposition 4.5]{cvop2014}   \label{prop:PS}
		If $(x^*,z^*) \in \R^{n+1}$ is an optimal solution of problem \eqref{PS}, then $ x^*$ is a weak minimizer. Moreover, $v + z^* d \in \bd \mathcal{P}$. 
	\end{proposition}
	
	\begin{remark}\label{prop:inc}
		Note that $f(x) \leq_C v+zd$ holds for some $x \in \X$ if and only if $v+zd \in \mathcal{P}$. 	To see, assume $f(x) \leq_C v+zd$ holds for some $x \in \mathcal{X}$, that is, $v+zd-f(x) \in C$ holds. Then, we have $v+zd \in \{f(x)\}+C \subseteq f(\mathcal{X})+C = \mathcal{P}$. The other implication follows similarly.
	\end{remark}
	The following propositions show the existence of an optimal solution to problem \eqref{PS} under the assumptions of problem \eqref{P} together with some additional conditions on the problem parameters. Note that \Cref{prop:PS2} is already stated in \cite[Proposition 4.4]{cvop2014}. We still provide a proof of it as the one given in \cite{cvop2014} has inaccuracies.\revv{\footnote{\revv{In the proof of \cite[Proposition 4.4]{cvop2014}, the feasible region of \eqref{PS} is stated to be compact.}}}
	\begin{proposition} \label{prop:PS2}
		Let $v \in \R^p$. If $d \in \Int C$, then there \rev{exist optimal solutions to~\eqref{PS} and \eqref{PS-dual}. Moreover, the optimal values of the two problems coincide.} 
	\end{proposition}	
	\begin{proof}
		First, we show that there exists a feasible solution to \eqref{PS}. \rev{By the assumptions on \eqref{P}, there exists  $\bar{x}\in \Int\X$. Let $\bar{z}:= \max_{i\in\{1,\ldots,l\}} \frac{(w^i)^\T(f(\bar{x})-v)}{(w^i)^\T d} + \delta$ for some $\delta >0$, where $C^+ = \cone \conv\{w^1,\ldots,w^l\}$. Note that $\bar{z}$ is well defined since $d \in \Int C$, hence $(w^i)^\T d > 0$ for all $i\in\{1,\ldots,l\}$. It is not difficult to see that $(\bar{x},\bar{z})$ is a feasible solution for \eqref{PS}. Indeed, $(\bar{x},\bar{z})$ is strictly feasible, hence a Slater point for \eqref{PS}. \revv{Hence,} there exists an optimal solution $w^*$ for \eqref{PS-dual} and the optimal values of the two problems coincide.}
		
		
		\rev{For the existence of an optimal solution for \eqref{PS}, we first} show that for any feasible solution $(x,z) \in \R^{n+1}$, $z$ is bounded below. As $\mathcal{X}$ is compact, $\eqref{P}$ is bounded, \rev{that is, t}here exists $a \in \R^p$ such that $\mathcal{P} \subseteq \{a\}+ C$. Then, by \Cref{prop:inc}, $v+zd \in \{a\}+ C$ holds for any feasible $(x,z)$. This implies that for all $w \in C^+$ we have $w^\T(v+zd-a) \geq 0$. Since $w^\T d> 0$ \rev{for all $w \in C^+\setminus\{0\}$}, it is true that $$ z \geq \sup_{w \in C^+\rev{\setminus \{0\}}}\frac{w^\T (a-v)}{w^\T d} = \rev{\sup \left\{\frac{w^\T (a-v)}{w^\T d} \mid w \in C^+, \norm{w} =1 \right\} } =:\tilde{z} \in \R.$$ Then, \rev{without changing the problem, we may} add the constraint: $\tilde{z} \leq z \leq \bar{z}$ to \eqref{PS}. 
		As the feasible region of the equivalent form is compact, \eqref{PS} has an optimal solution.
	\end{proof}


\begin{proposition}\cite[Proposition 3.7]{Dorfler2020}\label{prop:PS3}
	For $v \notin \mathcal{P}, y \in \rev{\Int}\mathcal{P}$ and $d=y-v$, \rev{both \eqref{PS} and \eqref{PS-dual} have optimal solutions and the optimal values coincide}.
\end{proposition}

	\rev{The next result states} that using the primal-dual solution pair to problem\rev{s \eqref{PS} and \eqref{PS-dual}}, it is possible to find a supporting hyperplane to the upper image. \rev{It will play a critical role in the description of the solution algorithm which will be discussed in Section \ref{ch:algorithm_variants}.}
	
	\begin{proposition}[{\cite[Proposition 4.7]{cvop2014}}]\label{prop:supphyp} Let $v, \rev{d} \in \R^p$ and $(x^*,z^*)$, $w^*\in \R^p$ be the optimal solutions for \eqref{PS} and its Lagrange dual, respectively. \rev{If the optimal objective \revv{function} values of \eqref{PS} and \eqref{PS-dual} coincides,} then $H:=\{y \in \R^p \mid (w^*)^\top y=(w^*)^\T v+z^*\}$ is a supporting hyperplane for $\P$ at $y^*=v+z^*d$ and $\mathcal{H}=\{y \in \R^p \mid (w^*)^\T y \geq (w^*)^\T v+z^*\}$ contains $\P$.
	\end{proposition}
		
	\rev{\Cref{fig:PS}  illustrates problem \eqref{PS} and Proposition \ref{prop:supphyp}.} 
	
	\begin{figure}[!]
	\centering
	\caption{\rev{Illustration of \eqref{PS} and \Cref{prop:supphyp}. The blue marker shows the point $v + z^\ast d$ and the dashed red line shows the supporting hyperplane $\mathcal{H}$ of $\P$ at $v + z^\ast d$. } }
	\includegraphics[width=0.4\textwidth]{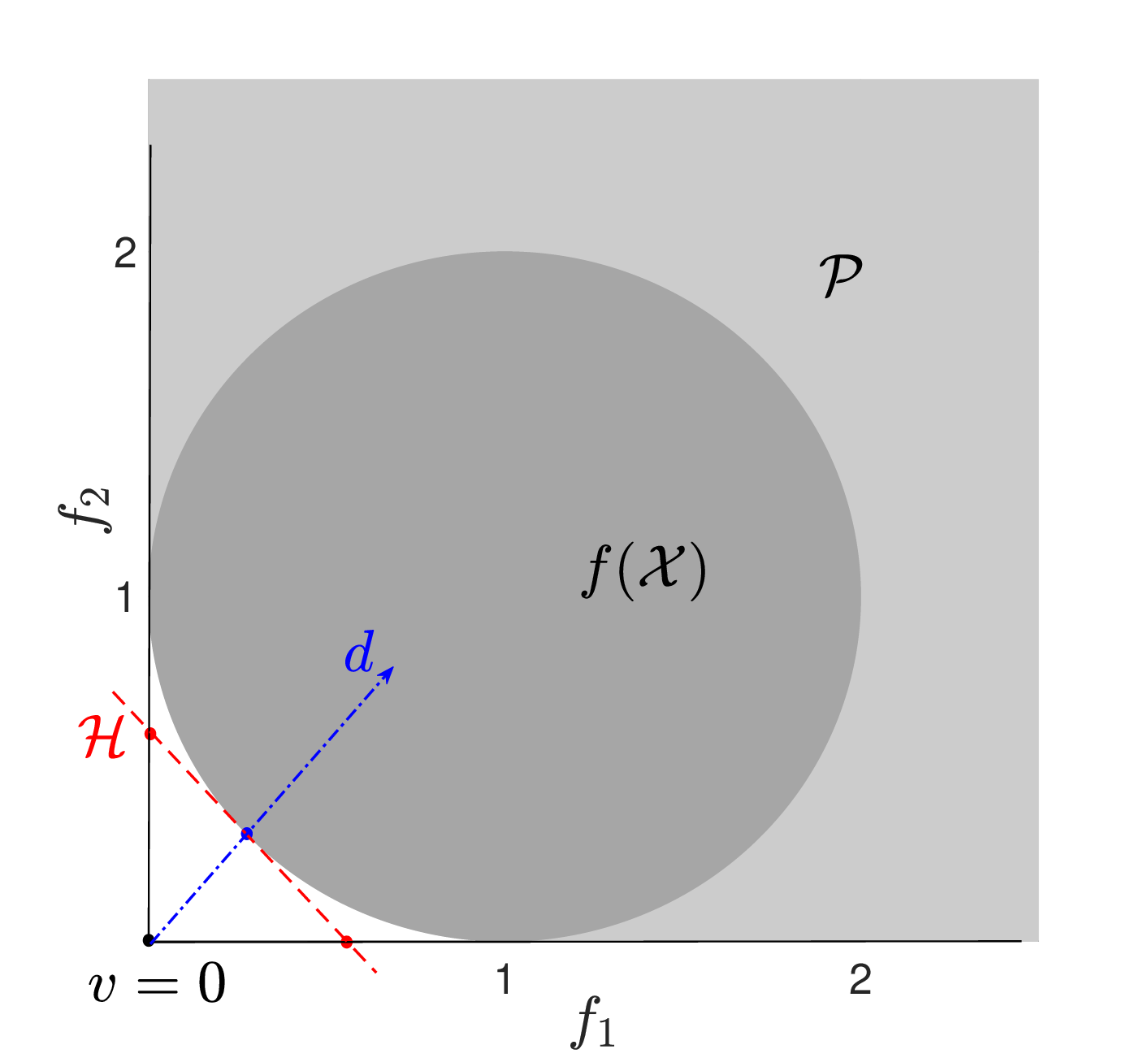}
	\label{fig:PS}
\end{figure}
	
	\section{The algorithm and variants}\label{ch:algorithm_variants}
	
	\rev{In this section,} we will first provide the main framework of a CVOP algorithm \revv{that} solves PS scalarizations iteratively and finds a finite weak $\epsilon$-solution to problem \eqref{P} for a given $\epsilon>0$. Then, we will provide some variants \revvv{based on selecting the parameters of the scalarizations}.
	
	\subsection{The algorithm}
	\rev{We consider a CVOP} algorithm that starts with an outer approximation to the upper image, iteratively updates it by solving PS scalarizations\revv{,} and stops when the approximation is fine enough. More specifically, it starts by solving (WS($w^i$)) for all $i\in\{1,\ldots,l\}$ where $w^i$'s are \rev{extreme directions} of $C^+$. Optimal solutions $x^i, i \in \{1,\ldots, l\}$ of these problems form the initial set of weak minimizers $\bar{\X}^0$. Then, the initial outer approximation $P^0$ of $\mathcal{P}$ is set as $P^0=\bigcap_{i \in \{1,\ldots,l\}}\mathcal{H}^i$ (lines 2, 3 of \Cref{alg_1}), where $\mathcal{H}^i=\{y \in \rev{\R^p} \mid {(w^i)}^{\T}y \geq {(w^i)}^{\T} f(x^i)\}$. 
	
	In \revv{the} $k^{th}$ iteration, the vertices $V^k$ of the current outer approximation $P^k$ \revv{are} considered and a vertex \revv{that was} not used in the previous iterations, \rev{i.e., $v \in V^k\setminus \Vused$} is selected. We will later discuss different vertex selection methods, some of which return $\Vinfo \neq \emptyset$. $\Vinfo$ stores triples $(v,y^v,z^v)$ for the vertices $v \in V^k \setminus \Vused$, where $y^v \in \bd\mathcal{P}$ and $z^v=\norm{y^v-v}$. In this case, an upper bound for the Hausdorff distance between $P^k$ and $\P$, namely $\hat{h}=\max_{v \in V^k \setminus \Vused}z^v$ can be computed. If $\hat{h} \leq \epsilon$, the algorithm terminates by letting $V^k= \emptyset$. $\Vinfo= \emptyset$ means that the vertex selection method does not store any information \rev{regarding} the current vertices, see lines 6-12.
	
	If $\Vinfo = \emptyset$ or if the algorithm is not terminated as explained above, then~\eqref{PS} is solved to find a weak minimizer $x^v$, see \Cref{prop:PS}. Note that if the direction parameter $d$ is not fixed from the beginning of the algorithm, then it has to be computed \rev{before this step}. 
	The selected vertex $v$ is added to $\Vused$ and the corresponding weak minimizer $x^v$ is added to set $\bar{\X}^k$ (lines 13-16). If the current vertex is close enough to the upper image, then the algorithm checks another vertex from $V^k \setminus \Vused$. Otherwise, the current outer approximation is updated by intersecting it with the supporting halfspace $\mathcal{H}$ of $\P$ at $f(x^v)$, see \Cref{prop:supphyp}. The vertices of the updated outer approximation \revv{are} computed by solving a vertex enumeration problem (lines 18-20). The algorithm terminates if all the vertices of the current outer approximation are in $\epsilon$ distance to the upper image and returns a set of weak minimizers $\bar{\X}$. 
	\begin{algorithm}[h]
		\caption{Primal Approximation Algorithm for (P)}
		\begin{algorithmic}[1] \label{alg_1}	
			\STATE $\Vused=\emptyset$, $k=0$, $solve=1$;
			\STATE For $i=1,\ldots,l$: Solve (WS($w^i$)) to find an optimal solution $x^i$;
			\STATE Set $\bar{\X}^0 = \{x^1,\ldots,x^l\}$ and $P^0 = \bigcap_{i \in \{1,\ldots, l\}} \{y \in \rev{\R^p} \mid (w^i)^{\T}y\geq (w^i)^{\T} f(x^i)\}$;
			\STATE Compute the set of vertices $V^0$ of $P^0$; 
			\STATE Fix rules for selecting a vertex from $V^k$ and a direction parameter $d$;
			\WHILE {$V^k \setminus \Vused \neq \emptyset$}
			\STATE  $[v, \Vinfo]\gets $SelectVertex($V^k, \Vused$); 
			\IF {$\Vinfo \neq \emptyset $}
			\IF {$\hat{h}=\max_{v \in V^k \setminus \Vused} z^v \leq \epsilon$}
			\STATE $V^k=\emptyset$, $solve=0$;
			\ENDIF
			\ENDIF
			\IF {$\Vinfo= \emptyset$ or $solve=1$}
			\STATE (If not fixed) compute $d^v$ such that $d^v \in \Int C$ and $\norm{d^v} = 1$;
			\STATE Solve (PS($v,d^v$))\rev{/(D-PS($v,d^v$))}. Let $(x^v,z^v)$ and \rev{$w^v$} be optimal solutions;
			\STATE $\Vused \gets \Vused \cup \{v\}$, $\bar{\X}^{k} \gets \bar{\X}^{k} \cup \{x^v\}$;
			\IF {$z^v > \epsilon$}
			\STATE \rev{Compute $\mathcal{H}=\{y\in \R^p \mid (w^v)^\T y \geq (w^v)^\T v + z^v\}$};
			\STATE $P^{k+1}\gets P^{k} \cap \mathcal{H}$, $\bar{\X}^{k+1} \gets \bar{\X}^{k}$;
			\STATE Compute the set of vertices $V^{k+1}$ of $P^{k+1}$; 
			\STATE $k \gets k+1$;
			\ENDIF
			\ENDIF
			\ENDWHILE
			\RETURN $\bar{\X}^k$
		\end{algorithmic}
	\end{algorithm} 
	\begin{remark}\label{rem:cardinality}
		\rev{In Algorithm \ref{alg_1},} to obtain a coarser set of solutions, instead of adding all weak minimizers to the solution set, one can add the ones which satisfy $z^v \leq \epsilon$. 
	\end{remark}
	
	Later, we will discuss different rules for selecting the direction parameter and the vertices, respectively in Sections \ref{sect:var_direction} and \ref{sect:var_vert}. The following proposition holds for any selection rule. 
	
	\begin{proposition} \label{prop:alg}
		When terminates, \Cref{alg_1} returns a finite weak $\epsilon$-solution.
	\end{proposition}
	
	\begin{proof}
		There exists optimal solutions \rev{$x^i$} to (WS($w^i$)) for all $i \in \{1, \ldots, l\}$ as $\X$ is compact and $f$ is continuous. \rev{Moreover, $\mathcal{H}^i=\{y\in\R^p\mid (w^i)^\mathsf{T}y \geq (w^i)^\mathsf{T}f(x^i)\} \supseteq \P = f(\X) + C$ holds since $x^i$ is an optimal solution for \revv{(WS($w^i$))} and $\inf_{c\in C}(w^i)^\T c = 0$. Then, $P^0\supseteq \P$. Moreover,} $P^0$ has at least one vertex as $C$ is pointed and \eqref{P} is bounded, see \revv{\cite[Corollary 18.5.3]{rockafellar}}. \rev{Hence, we have $V^0 \neq \emptyset$ (consequently $V^0 \setminus \Vused \neq \emptyset$). By \Cref{prop:ws}, the initial solution set $\bar{\X}^0=\{x^1,\ldots,x^l\}$ consists of weak minimizers. See lines 2-4 of \Cref{alg_1}}\revv{.} 
		
		At iteration $k$, SelectVertex() returns a vertex $v \in V^k \setminus \Vused$ and $\Vinfo$. First, consider the case $\Vinfo=\emptyset$. As $d^v \in \Int C$ is ensured \rev{(line 14)}, by \Cref{prop:PS2}, there exists solutions \rev{$(x^v,z^v)$ and $w^v$} to (PS($v,d^v$)) \rev{and (D-PS$(v,d^v)$), respectively}. By \Cref{prop:PS}, $x^v$ is a weak minimizer. Hence, for any iteration $k$, $\bar{\X}^k$ is finite and contains only the weak minimizers \rev{(line 16)}. By \Cref{prop:supphyp}, \rev{$\mathcal{H}=\{y\in \R^p \mid (w^v)^\T y \geq (w^v)^\T v + z^v\} \supseteq \P$ (line 18). Since $P^0 \supseteq \P$ and $P^{k+1} = P^k \cap \mathcal{H}$ for all $k\geq 0$, we have $P^k \supseteq \P$ for all $k\geq 0$ through the algorithm.}

		
		The algorithm terminates when $V^k \setminus \Vused = \emptyset$. Assume this is the case after $K$ iterations. This suggests that each $v \in V^K$ is also an element of $\Vused$\rev{, that is, $V^K\subseteq \Vused$ (line 16)} and $z^{v} \leq \epsilon$ for all $v \in V^K$ at termination \rev{(line 17)}. To show that $\bar{\X}^K$ is a \rev{finite} weak $\epsilon$-solution of \eqref{P} \rev{in the sense of Definition \ref{defn:weakepssol}}, it is sufficient to have	\begin{equation}\label{eq:inc}
		P^K \subseteq \conv f(\bar{\X}^K) + C + B(0,\epsilon).
		\end{equation}
		Similar to \cite[Lemma 5.2]{umer2020}, one can show that the recession cone of $P^k$ is $C$ for any $k$. Hence, $P^K = \conv V^K +C$ holds true. \rev{Then, for any $p\in \revv{P^K}$, there exist $(\lambda^v)_{v \in V^K} \geq 0$ and $\bar{c} \in C$ such that $\sum_{v \in V^K}\lambda^v=1$ and $p = \sum_{v \in V^K}\lambda^vv+ \bar{c}$.}
			 On the other hand, for each \rev{$v \in V^K\subseteq \Vused$, there exist $x^v\in \bar{\X}^K$, $z^v \leq \epsilon$ such that} $(x^v,z^{v})$ is an optimal solution of (PS($v,d^v$)). In particular, there exists $c^v \in C$ such that $v+d^vz^{v}= f(x^v)+c^v$. These imply
		\begin{align*}
			p=\sum_{v \in V^K}\lambda^vv+ \bar{c}
			=&\sum_{v \in V^K}\lambda^v(f(x^v)+c^v-d^vz^{v})+\bar{c}\\
			=&\sum_{v \in V^K}\lambda^v f(x^v)+\sum_{v \in V^K}\lambda^v c^v+\bar{c}-\sum_{v \in V^K}\lambda^vd^vz^{v}\revv{.}
		\end{align*}
		Clearly, $\sum_{v \in V^k}\lambda^vf(x^v) \in \conv f(\mathcal{\bar{X}}^K)$ and $\sum_{v \in V^K}\lambda^v c^v+\bar{c} \in C$. Moreover, as $z^v \leq \epsilon$ and $\norm{d^v}=1$ for each $v\in V^K$, we have 
		$\sum_{v \in V^K}\lambda^v z^v d^v \in B(0,\epsilon)$. Hence, $p \in \conv f(\bar{\X}^K) + C + B(0,\epsilon)$ and this implies \eqref{eq:inc} as $p\in \revv{P^K}$ is arbitrary.
		
		For the vertex selection rules that give $\Vinfo \neq \emptyset$, assume $\hat{h}= \max_{v \in V^K \setminus \Vused} z^v \leq \epsilon$ holds for some $K$. Note that if there exists a vertex $v$ in $V^K \cap \Vused$, then $z^v \leq \epsilon$ has to be satisfied by the structure of the algorithm. Hence, for the vertices $V^K$ of $P^K$, we have \rev{$z^v \leq \epsilon $}. \revv{S}imilar steps for the previous case can be applied to show that \rev{$\bar{\X}^K$ is a finite weak $\epsilon$-solution of \eqref{P}.}
	\end{proof}
		
	\subsection{Direction selection rules}\label{sect:var_direction}
	
	\rev{As explained in Section \ref{ch:intro}, algorithms from \cite{cvop2014} and \cite{ehrgottSS2011} follow the framework given by Algorithm \ref{alg_1}.}
In \cite{cvop2014}, a fixed direction parameter $d$ is used within \eqref{PS} and in \cite{ehrgottSS2011}, $d^v$ is taken as $\hat{p}-v$ for a fixed point $\hat{p} \in \mathcal{P}$. Before proceeding with the proposed direction selection rules, we use \revvv{two examples (Examples 1 (for $p=2$) and 2 provided in \Cref{sect:examples})} to show that the selection of $d$ in \cite{cvop2014} and $\hat{p}$ in \cite{ehrgottSS2011} affect the performance of these algorithms significantly. The results are summarized in \Cref{table:fixdir}, in which we report the number of scalarization models (SC) and the \rev{CPU time} (T).\footnote{\revvv{See \Cref{sect:Comput_pre} for the computer and solver specifications.}} 
	
		\begin{table}[h]
		\caption{Results  for different $d$ and $\hat{p}$ values fixed for algorithms in \cite{cvop2014} and \cite{ehrgottSS2011}, respectively}
		\resizebox{\textwidth}{!}{	\begin{tabular}{c|cccc|cccc|ccc|ccc|}
				\cline{2-15}
				& \multicolumn{8}{c|}{{Example 1} ($p=2, \epsilon= 0.005)$}                                                                         & \multicolumn{6}{c|}{{Example 2} ($\epsilon= 0.05)$}                                                      \\ \cline{2-15} 
				& \multicolumn{4}{c|}{$d$ in \cite{cvop2014}}                            & \multicolumn{4}{c|}{$\hat{p}$ in \cite{ehrgottSS2011}}                            & \multicolumn{3}{c|}{$d$ in \cite{cvop2014}}           & \multicolumn{3}{c|}{$\hat{p}$ in \cite{ehrgottSS2011}}                         \\
				& \textbf{$\begin{pmatrix} 1\\1 \end{pmatrix}$} & \textbf{$\begin{pmatrix} 0.1\\1 \end{pmatrix}$} & \textbf{$\begin{pmatrix} 0.01\\1 \end{pmatrix}$} & \textbf{$\begin{pmatrix} 0.001\\1 \end{pmatrix}$} & \textbf{$\begin{pmatrix} 1\\1 \end{pmatrix}$} & \textbf{$\begin{pmatrix} 0.1\\1 \end{pmatrix}$} & \textbf{$\begin{pmatrix} 0.01\\1 \end{pmatrix}$} & \textbf{$\begin{pmatrix} 0.001\\1 \end{pmatrix}$} & \textbf{$\begin{pmatrix} 1\\1 \end{pmatrix}$} & \textbf{$\begin{pmatrix} 0.1\\1 \end{pmatrix}$} & \textbf{$\begin{pmatrix} 0.01\\1 \end{pmatrix}$} & \textbf{$\begin{pmatrix} 100\\100 \end{pmatrix}$} & \textbf{$\begin{pmatrix} 10\\1000 \end{pmatrix}$} & \textbf{$\begin{pmatrix} 10\\10000 \end{pmatrix}$} \\ \hline
				\multicolumn{1}{|c|}{SC}        & 29         & \textbf{27}           & 35            & 41             & \textbf{17 }        & 31           & 39            & 47             & \textbf{31}         & 56           & 304           & \textbf{30 }              & 178              & 741               \\
				\multicolumn{1}{|c|}{T}      & 10.10      & \textbf{9.94}         & 13.34         & 14.62          & \textbf{6.27 }      & 11.06        & 13.97         & 23.76          & \textbf{13.60}      & 18.19        & 93.81         & \textbf{10.52}            & 61.12            & 226.63            \\ \hline
		\end{tabular}}
		\label{table:fixdir}
	\end{table}

	Motivated \revv{by} these results, we propose \rev{two additional selection rules for the direction parameter to be used in PS scalarization.} 
	
	\paragraph{Adjacent vertices\revv{-}based approach (Adj).}
	When the double description method is used to solve the vertex enumeration problem (line 20 of \Cref{alg_1}), the set\rev{s} of adjacent vertices for each vertex \rev{are also returned}. We use this information to compute the direction parameter for a given vertex. In particular, for each vertex $v$, the normal direction of a hyperplane passing through the adjacent vertices of $v$ is computed. \rev{The main motivation for this selection rule is that the positions of the adjacent vertices depend on the curvature of the upper image as they are formed by the supporting hyperplanes of the upper image at points that are possibly close to the current vertex. This geometric intuition is illustrated in \Cref{fig:adj_motivation} for a two-dimensional setting.}

	Note that the adjacency list of a vertex may also contain extreme directions. If for a vertex $v$, we have an adjacent extreme direction $z$, then it is known that the line segment $\{v+rz \in \R^p \mid r \geq 0\}$ is a face of $P^k$. In these cases, we construct an artificial adjacent vertex by moving the current vertex along the adjacent extreme direction, that is, we take $v+z$. \rev{An example of this case can be seen in \Cref{fig:adj_motivation}.} 
	
		\begin{figure}[!]
		\centering
		\caption{\rev{Three vertices of the current outer approximation are $v^1, v^2$ and $v^3$. The adjacent vertices of $v^1$ ($v^2$ and $v^1+e_2$) and $v^2$ ($v^1$ and $v^3$) and the directions based on the Adjacent vertices approach from vertices $v^1$ and $v^2$ are illustrated.} }
		\includegraphics[width=0.8\textwidth]{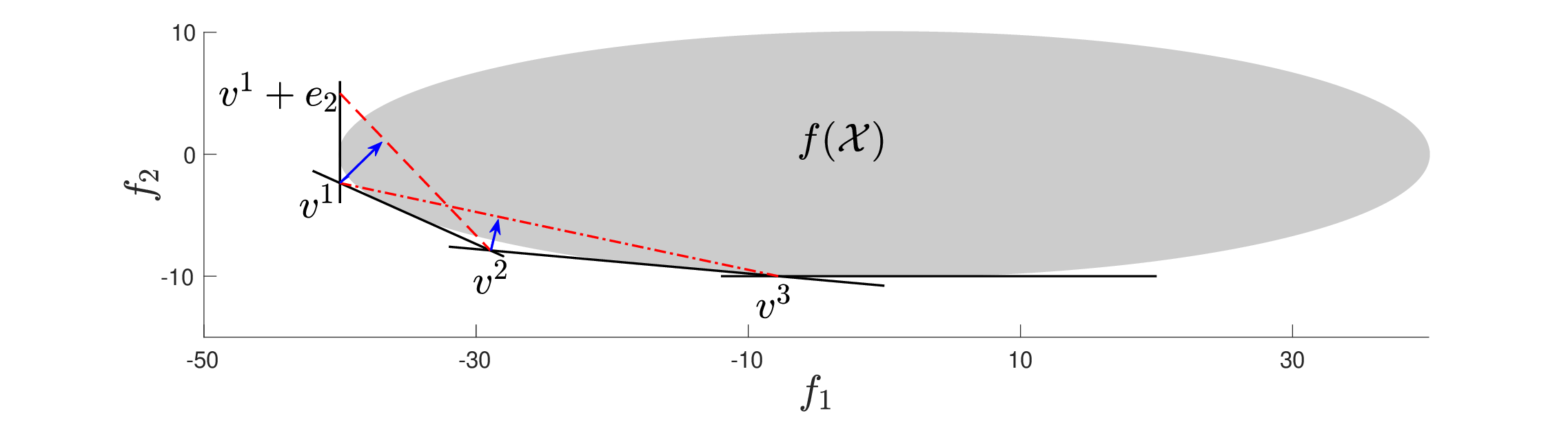}
		\label{fig:adj_motivation}
	\end{figure}
	
	\rev{In general,} a vertex may have more adjacent vertices than required. In that case, we choose $p$ \rev{of the vertices, $v^{i_1}_{adj},\ldots, \ v^{i_p}_{adj}$ and compute $d \in \R^p$ such that $d^\T v^{i_1}_{adj} = \ldots = d^\T v^{i_p}_{adj}$, see} 
	\revv{line 3} of Procedure \ref{proc:adj_dir}. \revv{Note that it is possible to obtain $-d$ as the solution of this system. On the other hand, by} \Cref{prop:PS2}, if $d \in \Int C$ we obtain a weak minimizer by solving \eqref{PS}. Hence, we first check if one of two candidate unit normal vectors\revv{, $d$ or $-d$,} is in $\Int C$. If this \revv{is} not the case, we use a predetermined direction $d \in \Int C$ where $\norm{d}=1$. \revv{In particular}, we set $d=\frac{\sum_{i=1}^{r}c^i}{\norm{\sum_{i=1}^{r}c^i}}$, where $c^1, \ldots, c^r$ are the extreme directions of $C$.
	\begin{procedure}[h]
		\caption{ChooseDirection($v$,$P$)}
		\begin{algorithmic}[1]
			\STATE Let $v^1_{adj}, v^2_{adj}, \ldots, \ v^s_{adj}$ be the vertices adjacent to $v$;
			\STATE Pick $p$ \rev{vertices $v^{i_1}_{adj},\ldots, \ v^{i_p}_{adj}$ and let $A = [v^{i_1}_{adj},\ldots, \ v^{i_p}_{adj}] \in \R^{p\times p}$}; 
			\STATE Compute \rev{$d \in B(0,1)$ satisfying $A^\T d = e$ (}$d=\frac{A^{-\T}e}{\norm{A^{-\T}e}}$ \rev{if $A^{-1}$ exists.)};	
			\IF {$d \in \Int C$}
			\RETURN $d$
			\ELSIF {$-d \in \Int C$}
			\RETURN $-d$
			\ELSE
			\RETURN  \revv{$d=\frac{\sum_{i=1}^{r}c^i}{\norm{\sum_{i=1}^{r}c^i}}$, where $c^1, \ldots, c^r$ are the extreme directions of $C$.} 
			\ENDIF
		\end{algorithmic}\label{proc:adj_dir}		
	\end{procedure}

	\begin{remark}	For $p=2, C=\R^2_+$, it is easy to \rev{show} that either $d \in \Int C$ or $-d \in \Int C$ is satisfied \revv{(see \Cref{fig:adj_motivation} for an illustration).} However, this may not be true in general for \rev{$p \geq 3$} or for $C \neq \R^2_+$. Through our computational tests, we obtain some counterexamples for {$p=4$}. Even though we haven't encountered this situation for $p=3$, it is still possible that $\pm d \notin \Int \R^3_+$. \revv{See \Cref{fig:counterex} for an illustration}, in which vertex $v=0\in \R^3$ has 4 adjacent vertices: $v^1_{adj} = (-1, 1, 1)^\T,\: v^2_{adj} = (1, 1, -1)^\T, \: v^3_{adj} = (3.3, -2.2, -1.1)^\T, \: v^4_{adj} = (5, -4, 1)^\T$. (Note that $v,v^1_{adj},\ldots,v^4_{adj}$ are still extreme points of the set $\conv \{v,v^1_{adj},\ldots,v^4_{adj}\} + \R^3_+$. For illustrative purposes, the figure shows $\conv \{v,v^1_{adj},\ldots,v^4_{adj}\} $.) The normal direction $d$ of the hyperplane passing through $v^2_{adj},v^3_{adj},v^4_{adj}$ satisfies $d \notin \Int \R^3_+$ and $-d \notin \Int \R^3_+$. 
			\begin{figure}[h]
				\centering
				\caption{\rev{An example with $p = 3, C = \R^3_+$ for which $d$ computed in line 5 of Procedure \ref{proc:adj_dir} is not in $\pm \Int \R^3_+$. In particular, the normal direction of the hyperplane passing thorough $v^2_{adj}, v^3_{adj}, v^4_{adj}$ is $d = (0.2273, 0.1818, -0.5909)^\T$.}}
				\includegraphics[width=0.4\textwidth]{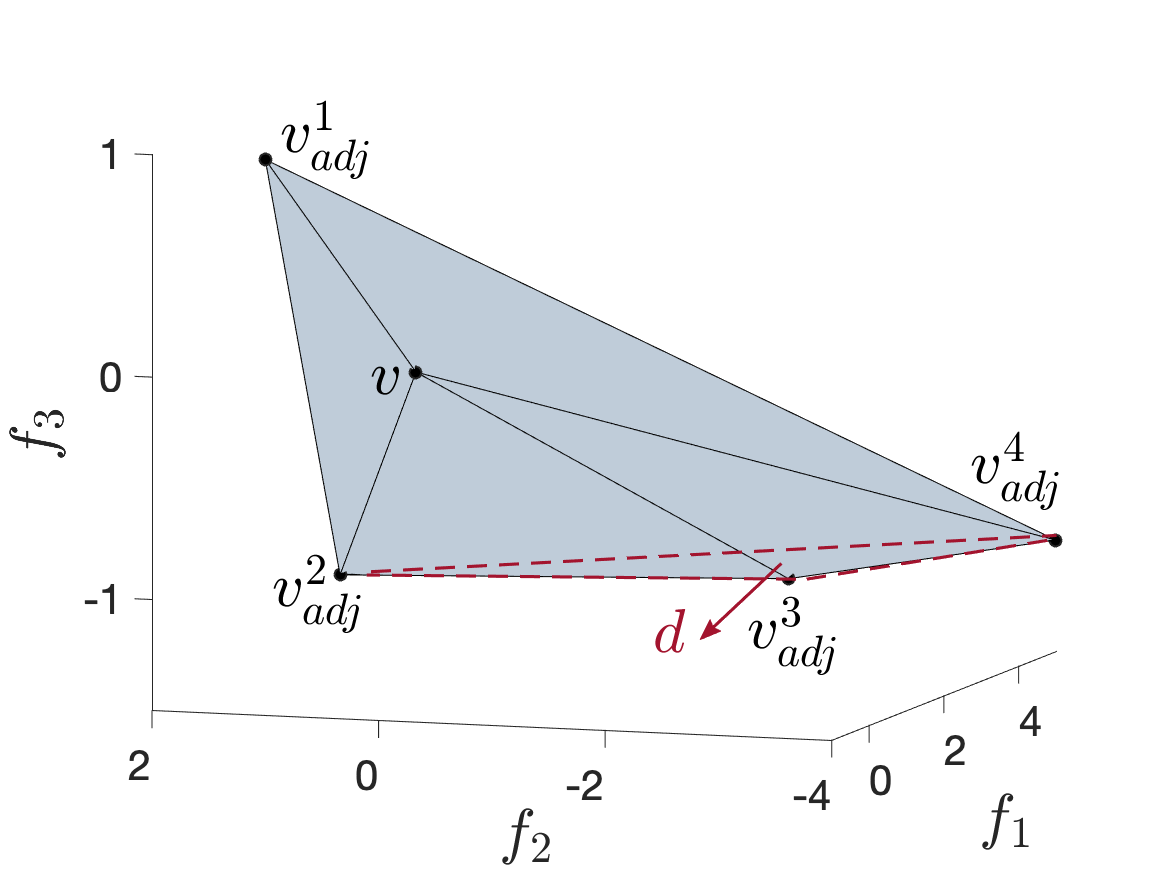}
				\label{fig:counterex}
			\end{figure}
	\end{remark}
	
	\paragraph{Ideal point\revv{-}based approach (IP).}
	For this approach, we assume that the ordering cone is $\R^p_+$, hence the ideal point $y^I$ is well defined. For a vertex $v$ of the current outer approximation, we consider the vector $v-y^I$. Note that for the initial iteration, we have $v-y^I=0$. In the subsequent iterations, we obtain $v-y^I \geq 0$ since $P^k \subseteq P^0$ for any $k$ throughout the algorithm. For $i \in  \{1,\ldots, p\}$, we define $d_i = \frac{1}{v_i - y^I_i+\bar{\epsilon}}$ for some sufficiently small $\bar{\epsilon}>0$, which is added for computational convenience. \rev{For the numerical examples, we take $\bar{\epsilon}=10^{-5}$.} Then $d$ is normalized such that $\norm{d} = 1$. Geometrically, this corresponds to considering the points $(v_i-y^I_i) e_i$ for each $i \in \{1,\ldots,p\}$ and \revv{taking} the normal direction of the hyperplane passing through them. 
	

	\subsection{Vertex selection rules}\label{sect:var_vert}
	
	We propose different vertex selection rules that can be used within \Cref{alg_1}. As will be detailed in \Cref{sect:alg_literature}, there are algorithms from the literature using some vertex selection rules that require solving additional optimization models, see \cite{Dorfler2020,Klamroth2002}. Our motivation is to propose vertex selection rules which are computationally less complicated yet have the potential to \rev{increase the efficiency of Algorithm \ref{alg_1}}. 
	
	\paragraph{Vertex selection with clusters (C).}
	This vertex selection rule clusters the vertices of the current outer approximation and visits these clusters sequentially. The motivation is to ensure that the vertices from different regions of the outer approximation are selected in a balanced fashion. 
	
	The first step is to fix centers for the clusters\revv{, which can be done in different ways.} {For this purpose, we first solve} \eqref{PS} for each vertex of the initial polyhedron $P^0$. The corresponding supporting halfspaces are intersected with the current polyhedron to obtain $P^1$. The same procedure is repeated for $P^1$ and the vertices of $P^2$ are selected as the centers of the clusters, see Procedure \ref{proc:centers}. \revv{It is possible to use the vertices of $P^1$ to have fewer clusters or continue in the same fashion and use the vertices of $P^k$ for some $k>2$ to have more clusters.} Note that SelectCenters($V^0$,$P^0$) has to be called right before line 6 of \Cref{alg_1}.
	
	\begin{procedure}[h]
		\caption{SelectCenters($V^0$,$P^0$)}
		\begin{algorithmic}[1] 
			\FOR {$k\in \{0,1\}$}
			\FORALL {$v \in V^k$}
			\STATE  Solve \eqref{PS}, let $(x^{v},z^{v})$ be the optimal solution, $\bar{\X} \gets \bar{\X} \cup \{x^{v}\}$;
			\STATE Find supporting halfspace $\mathcal{H}$ of $\mathcal{P}$ at $y^v=v+z^{v}d$, $P^{k+1} \gets  P^k \cap \mathcal{H}$;
			\ENDFOR
			\STATE Compute the set of vertices $V^{k+1}$ of $P^{k+1}$; 
			\ENDFOR
			\STATE $\mathcal{C}=\emptyset$.
			\FOR {$v \in V^{2}$}
			\STATE $\mathcal{C} \gets \mathcal{C} \cup \{v\}$;
			\ENDFOR
			\RETURN $\mathcal{C}$
		\end{algorithmic}\label{proc:centers}
	\end{procedure}

	For the remaining iterations of \Cref{alg_1}, each vertex of the current outer approximation is assigned to the cluster whose center is the closest with respect to the Euclidean distance. Whenever a vertex has to be selected, a vertex is chosen arbitrarily from the cluster in turn. If there is no unexplored vertex assigned to the current cluster, then the algorithm selects the next nonempty cluster. The pseudocode is given in Procedure \ref{VertCenters}.
	
	\begin{procedure}[h]
		\caption{SelectVertex($t, V^k, \Vused, \mathcal{C}$)}
		\begin{algorithmic}[1] 
			\STATE \rev{Initialize the $i^{th}$ cluster with center \rev{$c_i \in \mathcal{C}$} as $\mathcal{C}_i = \emptyset$ for $i\in \{1,\ldots, |\mathcal{C}|\}$;} 
			\FORALL {$v \in V^k \setminus \Vused$} 
			\STATE \rev{Pick $i \in \argmin_{i \in \{1,\cdots,|\mathcal{C}|\}} \norm{v-c_i}$} and $\mathcal{C}_i \gets \mathcal{C}_i \cup \{v\}$;
			\ENDFOR		
			\STATE \rev{$\textnormal{current}\equiv t+1 \pmod{|\mathcal{C}|}$};
			\WHILE {$\mathcal{C}_{\textnormal{current}} = \emptyset$}
			\STATE $t\gets t+1$ and $\textnormal{current}\gets t+1 \pmod{|\mathcal{C}|}$;
			\ENDWHILE
			\STATE Pick an arbitrary $v \in \mathcal{C}_{\textnormal{current}}$;
			\STATE $t\gets t+1$;
			\RETURN $v$, $t$
		\end{algorithmic} \label{VertCenters}
	\end{procedure}

	\paragraph{Vertex selection with adjacency information (Adj).}
	Recall that \emph{bensolve tools} returns the adjacency information for the vertices of the outer approximation. We use this to detect ``isolated" vertices of the current outer approximation. The motivation is to obtain uniformity among the vertices of the outer approximation and consequently, the $C$-minimal points found on the boundary of the upper image \revv{as long as the geometry allows}. For each vertex $v$, the procedure finds the minimum distance, say $dist_v$, from $v$ to its current neighbors. Then, it selects the vertex which has the maximum $dist_v$, see Procedure \ref{proc:ad}. 
	
	\begin{procedure}[h]
		\caption{SelectVertex($V^k$,$\Vused$)}
		\begin{algorithmic}[1] \label{proc:ad}
			\FORALL {$v \in V^k \setminus \Vused$}
			\STATE Let $\rev{V^v_{adj}}$ be the set of vertices adjacent to $v$;
			\STATE Let  $dist_v=\min_{\tilde{v} \in \rev{V^v_{adj}}} \norm{\tilde{v}-v}$;
			\ENDFOR
			\RETURN $v^* \in \argmax_{v \in V^k \setminus \Vused} dist_v$.
		\end{algorithmic}
	\end{procedure}
	\paragraph{Vertex selection using local upper bounds (UB).}
	This selection rule can only be used for $C = \R^p_+$. It is motivated by split algorithms, which are originally designed to solve multiobjective integer programming problems, see for instance \cite{Klamroth2,Holzman,Bektas}. The main idea is to find a set of local upper bounds for the nondominated points and use them to select a vertex. 
	
	For this method, in line 1 of \Cref{alg_1}, we additionally fix an upper bound $u=M e$ where $M$ is a sufficiently large number such that $\{u\}-\R^p_+ \supseteq f(\mathcal{X})$ and we initialize the set of upper bounds as $U=\{(u,\emptyset)\}$. Here, $\emptyset$ means that the upper bound $u$ is not defined based on any other (weakly) nondominated point found in the algorithm. Note that the initial upper bound $u$ satisfies that $y^I\leq u$. Through the algorithm, for any vertex $v$ of $P^k$, it is guaranteed that there exists a local upper bound $u$ such that $v\leq u$. Indeed, among all the upper bounds satisfying $v\leq u$, we fix the one that yields the minimum $\norm{v-u}$ value as the `corresponding' local upper bound for $v$.
	
	Together with the vertex $v$ to proceed, this method returns also the corresponding local upper bound $u$ where $(u,y)\in U$. Let $(x^v,z^v)$ be an optimal solution for \eqref{PS}.  Using $\rev{y^v}:=v+z^vd \in \bd \P$, $p$ new upper bounds $u^1,\ldots, u^p$ are generated as follows: for $j \in \{1, \ldots, p\}$, we set $u^j_i=u_i$ for each $i\in \{1, \ldots, p\}\setminus\{j\}$ and $u^j_j=y_j^v$. Then, the list of upper bounds is updated accordingly.  \rev{To capture these changes, line 7 of \revv{Algorithm} \ref{alg_1} is replaced by $$[v, (u,y), \Vinfo]\gets \textnormal{SelectVertex}(U, V^k, \Vused);$$ and} line 15 of \Cref{alg_1} is modified as \revv{in Procedure \ref{Createub}.}
	
			\begin{procedure}[h]
		\caption{{Line 15 of \Cref{alg_1} for UB}}
		
		\begin{algorithmic}[.]	\label{Createub}	
			\STATE \rev{Solve (PS($v,d$))/(D-PS($v,d$)). Let $(x^v,z^v)$/$w^v$ be optimal solutions and $y^v=v+z^vd$;}
			\FOR {$j = 1:p$}
			\STATE $u^j \gets u$,  $u^j_j \gets y^v_j$, $U \gets U \cup \{(u^j,y^v)\}$;
			\ENDFOR
			\STATE  $U \gets U \setminus \{(u,y)\}$;
		\end{algorithmic}		
	\end{procedure}
	For the vertex selection procedure, we first assign the corresponding upper bounds for each vertex (line 3 of Procedure \ref{vertex_ub}). Then, we choose $v^\ast \in V^k \setminus \Vused$, which has the greatest distance to its corresponding local upper bound $u^\ast$ (lines 10 and 12 of Procedure \ref{vertex_ub}). Note that $\norm{v^*-u^*}$ yields an upper bound for the current approximation error since we have $d_H(P^k,\mathcal{P}) = \max_{v \in V^k} d(v,\P) \leq \max_{v\in V^k} \norm{v-u^v} = \norm{v^*-u^*},$ where $u^v$ denotes the corresponding upper bound of $v$.\footnote{\revvv{In addition to the ones from \Cref{sect:var_direction}, different direction selection} \revv{ methods that are specifically designed for UB are discussed and tested. The details are provided in the appendix.}} As it will be detailed in \Cref{sect:alg_literature}, the algorithms proposed in \cite{Klamroth2002} and \cite{Dorfler2020} also compute an upper bound for the approximation error during any iteration. Different from them, UB does not solve optimization models to find this approximation error.
	
	Note that some upper bounds have components equal to $M$. This causes the distances between $v$ and its corresponding upper bound to be significantly large. To tackle this, the modifications are made, see lines 4-9 of Procedure \ref{vertex_ub} for the details.
	
	\begin{procedure}[h]
		\caption{SelectVertex($U$,$V^k, \Vused$)}
		\begin{algorithmic}[1] \label{vertex_ub}	
			\STATE $\Vinfo=\emptyset$	
			\FORALL {$v \in V^k \setminus \Vused$}
			\STATE Let $(\hat{u},\hat{y}) \in \argmin_{\{(u,y)\in U \colon v\leq u\}} \norm{v-u}$; 
			\STATE $u^{\textnormal{temp}}=\hat{u}$;
			\FOR {$i =1\colon p$}
			\IF {$\hat{u}_i = M$}
			\STATE $u^{\textnormal{temp}}_i \gets \max\{\textnormal{avg},\hat{y}_i\}$; \footnotemark
			\ENDIF
			\ENDFOR
			\STATE $\hat{z} = \norm{v-u^{\textnormal{temp}}}, \Vinfo \gets \Vinfo \cup \{(v,\hat{u}, \hat{z})\}$;
			\ENDFOR
			\STATE Let $(v^\ast, u^\ast, z^\ast) \in \argmax_{(v,u,z) \in \Vinfo } z$, and let $y^\ast$ be such that $(u^\ast,y^\ast)\in U$;
			\RETURN $v^\ast, (u^\ast,y^\ast), \Vinfo$		
		\end{algorithmic}		
	\end{procedure} \footnotetext{ avg is the average of $\hat{u}_i$ values that are not equal to $M$, that is $\textnormal{avg}=\frac{\sum_{i\colon \hat{u}_i \neq M} \hat{u}_i}{|\{i: \hat{u}_i \neq M\}|}$. Note that $u^{\textnormal{temp}}$ is in $\P$ by construction and it is used only to compute a better upper bound for $d(v,\P)$.}
	
	\begin{remark}\label{rem:KTW} 
		For Procedure \ref{vertex_ub} to work correctly, we modify \Cref{alg_1} slightly as follows: In line 1, we initialize $V^0_{\textnormal{info}}$ as an empty set and in line 19 we also update $V^{k+1}_{\textnormal{info}} \gets \Vinfo$. \revvv{The same modification is also applied to the algorithms from the literature that will be explained in \Cref{sect:alg_literature}. }
	\end{remark}

	\subsection{Algorithms from the literature} \label{sect:alg_literature}
	In this section, we briefly explain similar approaches from the literature, namely, Algorithm KTW from \cite{Klamroth2002}, Algorithm DLSW from \cite{Dorfler2020}, and \rev{an algorithm from \cite{umer2020}.\footnote{\revvv{The pseudocodes of the corresponding procedures are provided in the appendix.}} The algorithm from \cite{umer2020} does not totally fit into the framework of \Cref{alg_1} as it solves a different scalarization model in each iteration. Here, we consider a variant of \Cref{alg_1}, which is motivated by the one from \cite{umer2020}.}
	
\paragraph{Algorithm KTW.}
The outer approximation algorithm provided in \cite{Klamroth2002} is a special case of the general framework given by \Cref{alg_1}, applied to a maximization problem. Here, we shortly explain it for problem~\eqref{P}. KTW assumes $0\in\mathcal{P}$. 
In iteration $k$, it minimizes the distance (with respect to an oblique norm
) between the current outer approximation and the upper image. To do that, for every vertex $v \in V^k \setminus \Vused$, it solves:
\begin{align} \label{eq:GB}
\textrm{maximize } \lambda \:\: \textrm{subject to } f(x) \leq \lambda v, \: x \in \X,\: \lambda \in \R. 
\end{align}		
Assuming $(x^v,\lambda^v)$ is an optimal solution to \eqref{eq:GB}, 
$v^* \in \argmin_{v \in V^k \setminus \Vused} \lambda^v$ is selected in \Cref{alg_1}.

Problem~\eqref{eq:GB} is equivalent to~\eqref{PS} with $d = -v$. Indeed, $(x^v,\lambda)$ is an optimal solution to \eqref{eq:GB} if and only if $(x^{v},1-\lambda)$ is an optimal solution to (PS($v,-v$)). Note that KTW is similar to the algorithm proposed in~\cite{ehrgottSS2011} in the sense that $\hat{p}=0$ is fixed. However, different from~\cite{ehrgottSS2011}, the selection of the vertices in each iteration is not arbitrary in KTW. 

\rev{For the examples from Section \ref{sect:examples}, it is not guaranteed that $0\in \P$. To overcome this,} for the computational tests that will be presented in \Cref{ch:comp_results},  we take $\hat{p}$ as in \eqref{eq:phat}. Then, we modify the model given in \eqref{eq:GB} as:
\begin{align} \label{eq:GB2}
\textrm{maximize } \lambda \:\: \textrm{subject to }  f(x) \leq \hat{p}+\lambda(v-\hat{p}), \: x \in \X, \: \lambda \in \R.
\end{align}	
\rev{This simply corresponds to shifting the upper image as well as the vertices of the outer approximation by $\hat{p}\in \P$.}
To have a more efficient implementation to select the vertices, we do not solve \eqref{eq:GB2} for all $v \in V^k \setminus \Vused$. Instead, we check if \eqref{eq:GB2} is solved in previous iterations.

	\paragraph{Algorithm DLSW.}
Dörfler et al. \cite{Dorfler2020} propose a vertex selection rule that uses the inner approximation obtained in each iteration. Accordingly, 
for each vertex $v$ of the current outer approximation $P^k$, the following problem\rev{, which measures the distance from $v$ to the current inner approximation $\mathcal{I}^k:=\cl\conv f(\bar{\mathcal{X}}^{k})+C$,} is solved: 
\begin{align} \label{eq:QP}
\textrm{minimize } \norm{y-v}^2 \:\: \textrm{subject to } y \in \mathcal{I}^k.
\end{align}	
Then, \rev{vertex $v^\ast$ which is the farthest away from the inner approximation is selected, that is,} $v^\ast \in \argmax_{v \in V^k \setminus \Vused}\norm{y^v-v}$, where $y^v$ is an optimal solution of \eqref{eq:QP}. \rev{Moreover,} $d$ is set to $\frac{y^{v^\ast}-v^\ast}{\norm{y^{v^\ast}-v^\ast}}$. An upper bound for the Hausdorff distance between the current outer approximation and the upper image is found as $\norm{y^{v^\ast}-v^\ast}$.

In \cite{Dorfler2020}, also an improved version of this algorithm is presented. Using the information from the previous iterations, it is possible to skip solving \eqref{eq:QP} for some of the vertices. \rev{In particular, for $x^{k}$ being the weak minimizer found in iteration $k$, 
	if \eqref{eq:QP} is solved for some $v$ in the previous iterations and $(y^{v}-v)^\T (f(x^k)-y^{v}) \geq 0$ holds true, then the solution of \eqref{eq:QP} for $v$ in iterations $k$ and $k-1$ are the same by \cite[Theorem 4.4]{Dorfler2020}. Hence there is no need to solve \eqref{eq:QP} for $v$.} See \cite{Dorfler2020} for the details.
	
	\paragraph{Algorithm AUU.}
Ararat et al. \cite{umer2020} propose an outer approximation algorithm to solve \rev{CVOPs}. Even though \rev{this} algorithm is similar to \Cref{alg_1}, it is different since instead of \eqref{PS}, it solves the following scalarization: 
\begin{align} \label{eq:NM}
\textrm{minimize } \norm{z} \:\: \textrm{subject to } f(x) \leq_C v+z, \:  x\in \X, z\in \R^p.
\end{align}	
Note that \eqref{eq:NM} does not require a direction parameter. Instead, it computes the distance \rev{$z^v$} from $v$ to the upper image, \rev{where $(x^v,z^v)$ is an optimal solution}. 
If one solves \eqref{eq:NM} for all vertices of the current outer approximation $P^k$, then it \revv{is} possible to compute the exact Hausdorff distance between $P^k$ and $\mathcal{P}$ as $\max_{v\in V^k} z^v$, where $V^k$ is the set of vertices of $P^k$. See \cite{umer2020} for the details.

\rev{As the main motivation of this study is to compare the effect of parameter selection rules for the PS scalarization, we consider a variant (namely, Algorithm AUU) of \Cref{alg_1} which is motivated by the algorithm from \cite{umer2020}.} For  AUU, we solve \eqref{eq:NM} only for \rev{parameter selection for the PS scalarizations} and we still execute line 14 of \Cref{alg_1}. \rev{The direction parameter is set as $d=\frac{z^v}{\norm{z^v}}$. Note that this is the direction \revv{that} would yield the minimum optimal objective function value for \eqref{PS}. } 


In \cite{umer2020}, the vertex selection is arbitrary\rev{, however, a special vertex selection rule is discussed in \cite[Corollary 7.4 and Remark 7.5]{umer2020} to obtain a convergence result}. Here, for algorithm AUU, we \rev{apply this vertex selection rule, that is, we} solve \eqref{eq:NM} for each vertex of $P^k$ and select the one that is the farthest away from the upper image. This allows the algorithm to provide the approximation error in each iteration rather than at termination. Indeed, \rev{by Lemma \ref{lem:Hausdorffdist}}, this approximation error is equal to the Hausdorff distance between $\P$ and the outer approximation. As in KTW, the solution of \eqref{eq:NM} for a vertex $v$ is added to $\Vinfo$, if \eqref{eq:NM} is not solved for $v$ in the previous iterations. 

	\section{\revvv{Test examples and preliminary results}} \label{sect:Comput_pre}

Before we proceed \revvv{with the main computational study}, we provide the numerical examples that will be used throughout. \revvv{Moreover, to possibly decrease the number of variants to be tested for the main computational study, we provide some preliminary results.}

\revvv{For the computational tests provided here and in Section~\ref{ch:comp_results}, the algorithms} are implemented using MATLAB R2020b. The scalarizations are solved via CVX v2.2 \cite{cvx,gb08} and SeDuMi 1.3.4 \cite{s98guide}.  \revvv{Vertex enumeration problems are solved by \textit{bensolve tools} \cite{lohneWeissing2017}, which is a free solver providing commands for {the} calculus of convex polyhedra and polyhedral convex functions for Octave and MATLAB.} The computer specification that is used throughout the computational study is  Intel(R) Core(TM) i7-4790 CPU @ 3.60GHz.

\subsection{Test examples} \label{sect:examples}
Consider problem \eqref{P}, where the ordering cone $C$ is $\R^p_+$; and objective function $f:\R^n\to \R^p$ and feasible region $\mathcal{X}\subseteq \R^n$ are given as follows:

\begin{enumerate}
	\item[]\textit{Example 1.} $f(x) = x, \: \mathcal{X}=\{x \in \R^n \mid \norm{x-e} \leq 1, \: x \geq 0\}$ for $n=p \in \{\rev{2,}3,4\}$.
	\item[]\textit{Example 2.}\footnote{This example is bounded in the sense that the upper image is included in $y^I+\R^2_+$ where $y^I=0 \in \R^2$. In theory, the weighted sum scalarizations for the initialization step\rev{, namely, $(\inf_{x\in\R_+} x)$ and $(\inf_{x\in\R_+} \frac{1}{x})$ } do not have \rev{solutions as $x=0$ is not in the domain of the problem and $(\inf_{x\in\R_+} \frac{1}{x}=0)$ is not attained, respectively. 
			For the computational tests, we} manually \rev{set} the initial outer approximation as $P^0=\{0\}+\R^2_+$. 
	} $f(x)=\left(x,\: \frac{1}{x} \right)^\T, \: \mathcal{X} = \R_+$.
	\item[]\textit{Example 3.}\footnote{This set of examples \revv{is} taken from $\cite{Dorfler2020}$.} Four instances of the following example are considered for $a \in \{5,7,10,20\}$:  $$f(x) = x, \: \mathcal{X}=\left\{x \in \R^3 \big| \left(\frac{x_1-1}{1}\right)^2+\left(\frac{x_2-1}{a}\right)^2+\left(\frac{x_3-1}{5}\right)^2\leq 1\right\}.$$ 
	\item[]\textit{Example 4.} Three instances of the following example are considered for $a \in \{5,7,10\}$: $$f(x) = x, \: \mathcal{X}=\left\{x \in \R^4 \big| \left(\frac{x_1-1}{1}\right)^2+\left(\frac{x_2-1}{a}\right)^2+\left(\frac{x_3-1}{5}\right)^2 +\left(\frac{x_4-1}{1}\right)^2\leq 1\right\}.$$ 
	\item[]\textit{Example 5.} 
	$f(x)=P^\T x, \: \X=\{x \in \R^n \mid A^\T x \leq b\},$
	where $P \in \R^{p \times n}$,  $A \in \R^{m \times n}$ and $b \in \R^m$. 
	For the computational tests, each  component of $A$ and $P$ is generated according to independent normal distributions with mean 0, and variance 100; each component of $b$ is generated according to a uniform distribution between 0 and 10. Components are rounded to the nearest integer for computational simplicity.	
\end{enumerate}	

\revvv{We consider examples where the ordering cone is the positive orthant, mainly because some of the algorithms considered in our computational tests are designed to solve MOPs only. However, Algorithm \ref{alg_1} can solve vector optimization problems with ordering cones different from the positive orthant. To illustrate this feature, we consider Example 1 for $p\in \{2,3\}$ with different ordering cones taken from \cite{umer2020}. The details and results are provided in the appendix.}

\subsection{A preliminary computational study} \label{subsect:preliminary}
As \revv{a} preliminary analysis, we compare the proposed direction selection methods with the ones from \revv{the} literature. 

For the \emph{fixed point approach} (FP) from \cite{ehrgottSS2011}, $\hat{p} \in \mathcal{P}$ is set as 
\begin{align}\label{eq:phat}
\hat{p}_i := 2 \max\{f_{i}(x^1), \ldots, f_{i}(x^p)\}-v^0_i, 
\end{align} where \rev{$x^i$ is the} optimal solution of (WS($w^i$)) for $i \in \{1, \ldots, p\}$ found in the initialization step and $v^0$ is a vertex of the initial polyhedron $P^0$. \rev{Note that as the ordering cone is $\R^p_+$, we have $v^0 = y^I$. Moreover, it is not difficult to see that $\hat{p} \in \Int \P$ for all the test examples.}

The direction parameter for a vertex $v$ is set as $d= \frac{\hat{p}-v}{\norm{\hat{p}-v}}$. By \Cref{prop:PS3}, there are optimal solutions to \eqref{PS} and \eqref{PS-dual} under this direction selection rule. However, for the computational tests, we ensure $d \in \Int \R^p_+$ for computational simplicity and to reduce the numerical issues. Accordingly, if $d \notin \Int \R^p_+$ then we use \revv{the} predetermined direction $d$ as in (Adj)\revv{, which gives $d = \frac{e}{\norm{e}}$ since $C=\R^p_+$.} For the \emph{fixed direction approach} (FD) from \cite{cvop2014}, $d$ is \revv{also} set to $\frac{e}{\norm{e}}$.

\revvv{For our preliminary analysis, } we employ different ``simple" vertex selection rules which are summarized below. 

\begin{itemize}
	\item \textbf{Closest to the Ideal Point (Id)}\label{Id}: Assume that the ordering cone is $\R^p_+$ and the ideal point $y^I$ is well-defined. Choose $v \in V^k \setminus \Vused$ such that  $ \left\lVert v-y^I \right\rVert$ is minimum.
	\item \textbf{\rev{Farthest} to an Inner Point (In)}\label{In}: 
	\rev{Choose $v \in V^k \setminus \Vused$ which yields the greatest distance  $\norm{\hat{p}-v}$}, where $\hat{p}$ is computed as in \eqref{eq:phat}. 
	\item \textbf{Random Choice (R)}\label{R}: Randomly choose a vertex among $V^k \setminus \Vused$, using \revv{the} discrete uniform distribution. \rev{For the numerical examples, we run the same example five times for this vertex selection rule.}
\end{itemize}

For the test examples in \Cref{sect:examples}, we run Algorithm \ref{alg_1} for direction selection rules FD, Adj, IP, FP and vertex selection rules Id, R, In. \rev{For the total of 12 variants formed by these parameter selection rules,} we report the total number of scalarization models solved (SC) as well as the total CPU time (T) required, see \Cref{tab:prelim}. The results for Example~5 display the averages over 20 randomly generated problem instances.

\begin{table}[htbp]
	\centering
	\caption{Preliminary computational results \rev{for Examples 1-5. The columns and rows show the direction and simple vertex selection rules discussed in Section \ref{sect:var_direction}, respectively. Each instance is solved in 12 different settings and the smallest SC and T values are indicated by boldface.}\footnotemark}
	\resizebox{\textwidth}{!}{
		\begin{tabular}{|c|cc|cc|cc|cc||cc|cc|cc|cc|}
			\cline{2-17}    \multicolumn{1}{r|}{} & \multicolumn{2}{c|}{FD} & \multicolumn{2}{c|}{Adj} & \multicolumn{2}{c|}{IP} & \multicolumn{2}{c||}{FP} & \multicolumn{2}{c|}{FD} & \multicolumn{2}{c|}{Adj} & \multicolumn{2}{c|}{IP} & \multicolumn{2}{c|}{FP} \\ \cline{2-17}
			\multicolumn{1}{r|}{} & SC    & T  & SC    & T  & SC    & T & SC    & T & SC    & T  & SC    & T & SC    & T  & SC    & T \\
			\cline{2-17}    \multicolumn{1}{r|}{} & \multicolumn{8}{c||}{Example 1, $p=3, \epsilon=0.005$}               & \multicolumn{8}{c||}{Example 1, $p=4, \epsilon=0.05$} \\
			\hline
			Id    & 488   & 134.73 & 425   & 116.76 & 627   & 169.90 & 461   & 126.79 & 900   & 275.15 & 1099  & 328.62 & 600   & 177.79 & 682   & 205.87 \\
			R     & 461.2 & 127.60 & 392.2 & 107.55 & \textbf{388}   & \textbf{105.51 }& 407.8 & 111.29 & 460   & 153.96 & 460.6 & 140.05 & 420   & \textbf{126.76} & \textbf{416 }  & 128.83 \\
			In    & 502   & 134.40 & 420   & 111.77 & 457   & 118.89 & 436   & 114.06 & 755   & 232.26 & 605   & 183.87 & -     & -     & -     & - \\
			\hline
			\multicolumn{1}{r|}{} & \multicolumn{8}{c||}{Example 2, $\epsilon=0.005$}                    &       & \multicolumn{1}{r}{} &       & \multicolumn{1}{r}{} &       & \multicolumn{1}{r}{} &       & \multicolumn{1}{r}{} \\
			\cline{1-9}    Id    & 154   & 35.38 & 134   & 31.11 & 141   & 31.83 & 283   & 63.99 &       & \multicolumn{1}{r}{} &       & \multicolumn{1}{r}{} &       & \multicolumn{1}{r}{} &       & \multicolumn{1}{r}{} \\
			R     & 127.8 & 27.64 & 120.8 & 26.54 & \textbf{113}   & \textbf{24.49} & 223.6 & 48.40 &       & \multicolumn{1}{r}{} &       & \multicolumn{1}{r}{} &       & \multicolumn{1}{r}{} &       & \multicolumn{1}{r}{} \\
			In    & 222   & 48.48 & 165   & 35.26 & 131   & 28.74 & 271   & 58.93 &       & \multicolumn{1}{r}{} &       & \multicolumn{1}{r}{} &       & \multicolumn{1}{r}{} &       & \multicolumn{1}{r}{} \\ \hline
			\multicolumn{1}{r|}{} & \multicolumn{8}{c||}{Example 3, $a = 5, \epsilon = 0.05$}             & \multicolumn{8}{c|}{Example 3, $a = 7, \epsilon = 0.05$} \\ \hline
			Id    & 142   & 38.28 & 155   & 41.79 & 114   & 31.45 & 187   & 50.10 & 150   & 39.71 & 134   & 35.55 & 160   & 42.25 & 291   & 77.27 \\
			R     & 125.8 & 33.56 & \textbf{102.4} & \textbf{27.34} & 109   & 29.03 & 173.4 & 46.60 & 123.2 & 32.77 & \textbf{111.4} & \textbf{29.85} & 115.4 & 30.55 & 204.6 & 54.62 \\
			In    & 120   & 31.73 & 116   & 30.77 & 124   & 32.63 & 180   & 47.67 & 146   & 38.97 & 113   & 30.77 & 120   & 31.98 & 224   & 60.08 \\
			\hline
			\multicolumn{1}{r|}{} & \multicolumn{8}{c||}{Example 3, $a = 10, \epsilon = 0.05$}            & \multicolumn{8}{c|}{Example 3, $a = 20, \epsilon = 0.05$} \\
			\hline
			Id    & 150   & 39.95 & 123   & 33.70 & 130   & 34.69 & 273   & 74.12 & 204   & 55.04 & 141   & 38.26 & 148   & 39.87 & 517   & 144.02 \\
			R     & 138.2 & 36.78 & 118.8 & 31.64 & \textbf{108.2} & \textbf{28.66} & 250.6 & 67.15 & 162   & 43.81 & \textbf{128.6} & \textbf{34.86} & 147.2 & 39.43 & 472.8 & 130.10 \\
			In    & 156   & 41.21 & 136   & 35.58 & 129   & 34.25 & 290   & 76.92 & 153   & 41.56 & 130   & 35.09 & 144   & 38.99 & 666   & 183.50 \\
			\hline
			\multicolumn{1}{r|}{} & \multicolumn{8}{c||}{Example 4, $a = 5, \epsilon = 0.05$}             & \multicolumn{8}{c|}{Example 4, $a = 7, \epsilon = 0.05$} \\ \hline
			Id    & 5529  & 2085.90 & 2497  & 630.53 & -     & -     & 7218  & 2806.85 & 8153  & 2538.08 & -     & -     & 1973  & 646.46 & 7099  & 2938.67 \\
			R     & 1185  & 424.59 & \textbf{1107.4} & \textbf{379.02} & -     & -     & -     & -     & 1247.4 & 450.86 & \textbf{1139 } & \textbf{397.09} & -     & -     & 3093.6 & 1567.57 \\
			In    & 2401  & 800.11 & 1584  & 526.82 & -     & -     & 5742  & 2203.63 & 2956  & 1012.39 & 4969  & 1554.33 & 3452  & 1013.43 & 6327  & 2771.90 \\ \hline
			\multicolumn{1}{r|}{} & \multicolumn{8}{c||}{Example 4, $a = 10, \epsilon= 0.05$}            &       & \multicolumn{1}{r}{} &       & \multicolumn{1}{r}{} &       & \multicolumn{1}{r}{} &       & \multicolumn{1}{r}{} \\
			\cline{1-9}    Id    & 3292  & 1120.34 & 2807  & 880.45 & 3050  & 937.37 & -     & -     &       & \multicolumn{1}{r}{} &       & \multicolumn{1}{r}{} &       & \multicolumn{1}{r}{} &       & \multicolumn{1}{r}{} \\
			R     & 1359.6 & 498.34 & \textbf{1150.6} & \textbf{396.60} & -     & -     & -     & -     &       & \multicolumn{1}{r}{} &       & \multicolumn{1}{r}{} &       & \multicolumn{1}{r}{} &       & \multicolumn{1}{r}{} \\
			In    & -     & -     & 2098  & 710.81 & 2106  & 683.75 & 11876 & 5970.35 &       & \multicolumn{1}{r}{} &       & \multicolumn{1}{r}{} &       & \multicolumn{1}{r}{} &       & \multicolumn{1}{r}{} \\
			\hline 
			\multicolumn{1}{r|}{} & \multicolumn{8}{c||}{Example 5, $d = 2, \epsilon = 10^{-8}$    }       & \multicolumn{8}{c|}{Example 5, $d = 3, \epsilon= 10^{-8}$} \\
			\hline
			Id    & 73.8  & 13.93 & 73.3  & 13.86 & 71.0  & 13.41 & 72.8  & 13.73 & 308.5 & 54.64 & 301.3 & 53.55 & 292.5 & 52.04 & 357.0 & 62.89 \\
			R     & 71.6  & 13.74 & 72.5  & \textbf{13.39} & \textbf{70.8 } & 13.47 & 71.3  & 13.46 & 192.1 & 34.97 & 194.8 & 35.50 & \textbf{186.2} & \textbf{33.91} & 200.5 & 36.49 \\
			In    & 75.2  & 14.12 & 72.2  & 13.64 & 71.1  & 13.44 & 73.2  & 13.87 & 337.0 & 59.45 & 284.0 & 50.56 & 294.2 & 51.99 & 327.6 & 57.59 \\
			\hline
	\end{tabular}}%
	\label{tab:prelim}%
\end{table}%
\footnotetext{For some of the instances, \textit{bensolve tools} was unable to perform vertex enumeration due to numerical issues. These are indicated by (-).}

We observe that while \revv{no direction selection method consistently outperforms the others} for the same vertex selection method, FD and FP are outperformed by either Adj or IP in almost all cases except two: In Example 1 with $p=4$ and vertex selection rule R, FP has the smallest SC; and in Example 4 with $a=7$, FD has the smallest SC and T.

In Examples 2-4, FP consistently performs worse than the other direction selection rules in terms of SC and T. On the other hand, \revv{in} more than half of all the cases (21 out of 36) Adj performs the best among other direction selection rules, whereas IP performs the best in around one\revv{-}third of all the cases (13/14 out of 36). Based on these results, 
we fix Adj as \rev{the} direction selection rule for the main computational analysis. 

\rev{When we compare the vertex selection rules for a given direction selection rule, we observe that} R performs the best in almost all instances. The reason that Id and In decrease the algorithm's performance may be the result of these methods' tendency to choose vertices that are close to each other, especially for certain problem instances. However, R also has a \rev{disadvantage} because it may not always provide consistent results. 

\section{Main computational results}
 \label{ch:comp_results}
	In this section, we present a computational study on the examples listed in \Cref{sect:examples}. We compare the proposed variants with the algorithms discussed in \Cref{sect:alg_literature}. We use three stopping criteria: the approximation error, CPU time, and cardinality of the solution set. 	
	
	Recall that depending on our preliminary analysis from \Cref{subsect:preliminary}, we fix the direction selection rule as Adj. On the other hand, we consider all vertex selection rules, namely C, Adj, and UB as introduced in \Cref{sect:var_vert}. In addition, we consider the random vertex selection, namely R, because of its good performance in our preliminary analysis in  \Cref{subsect:preliminary}. Since the direction selection rule is common, we simply call the proposed variants C, Adj, UB, and R. For R, we solve each problem five times and report the average values. 
	
	Before presenting our computational study, we discuss two proximity measures. The first one is the approximation error, which is the realized Hausdorff distance between the upper image and the final outer approximation, i.e., $d_H(P^K,\mathcal{P})$. Note that AUU returns the exact $d_H(P^K,\mathcal{P})$ by its structure. For all the other algorithms, we compute the correct Hausdorff distance using the vertices $v \in V^K$, after the termination. 
	
	The second measure is the hypervolume gap, which is originally designed in the context of multiobjective combinatorial optimization, {see for instance \cite{zitzler1999multiobjective},} and recently used for convex vector optimization \cite{Simay}. Here, we use it similar\revv{ly} to \cite{Simay}. Let \rev{$\mathcal{Q}\subseteq \revv{\R^p}$} 
	be such that \rev{$\mathcal{Q} \supseteq f(\X)$} 
	and $V^K$ be the vertices of the outer approximation at the termination of the algorithm. Let $\Lambda$ be the Lebesgue measure on $\R^p$ and $\Lambda_{\textnormal{in}}:=\Lambda((\conv f(\bar{\X}^K) +C) \cap \rev{\mathcal{Q}}), \Lambda_{\textnormal{out}}:=\Lambda((\conv V^K +C)  \cap \rev{\mathcal{Q}})$. Then, we compute the hypervolume gap as HG $=\Lambda_{\textnormal{out}}-\Lambda_{\textnormal{in}}$. 
	
	Noting that the ordering cone is \rev{$\R^p_+$} for all the examples from \Cref{sect:examples}, the \rev{set $\mathcal{Q}$ is taken as  $\{u\}-C$, where} upper bound $u$ is set such that $u_i := \max_{x\in\X} f_i(x)$ 
	for $i=1,\ldots,p$.	For the examples that $u_i$ cannot be computed (e.g., Example 2), the upper bound is set as $u_i= \max f_i(\bar{\mathcal{X}}^{1}\cup\ldots \cup\bar{\mathcal{X}}^{s})$, where $s$ is the number of the variants of \Cref{alg_1} that we solve the problem and $\bar{\mathcal{X}}^{i}$ is the solution set returned by the $i^{th}$ variant. 
	\rev{For Example 2, the solutions found by the weighted-sum scalarizations at initialization are eliminated from $\bar{\mathcal{X}}^{i}$ when computing $u_i$. This is because these solutions are excessively large in one component and this causes \emph{bensolve tools} not to perform the vertex enumeration as a result of numerical issues.} 
	
	In the tables throughout, we report the total number of models \rev{solved} to choose a vertex (VS), the number of scalarizations solved throughout the algorithm (SC), the cardinality of the solution set (Card), the CPU time (T), the realized approximation error (Err), and the hypervolume gap (HG). Note that VS is positive only for the algorithms in \Cref{sect:alg_literature} as the others do not solve models to select vertices. The realized approximation error is $d_H(P^K,\P)$.  For Example 5 with $p=3$, we generate 20 instances as explained in \Cref{sect:examples} and report the averages of the results; whereas\revv{,} for $p=4$, we generate 3 instances and report the results of each instance separately. We use \revv{the} convhulln() function in MATLAB to compute HG values. The ones that could not be calculated are given as (-) in the tables. 
	
	\subsection{Computational results based on \revv{the} approximation error}\label{sect:app_error}
	
	In this section, we solve the examples with a predetermined approximation error $\epsilon>0$, that is, when the algorithms terminate it is guaranteed that the Hausdorff distance between the outer approximation and $\P$ is less than $\epsilon$. Tables \ref{tab:ex_1-3} and \ref{tab:ex_4-5} show the results for Examples 1-3 and 4-5, respectively. Since 
	the realized errors are close to each other for all the algorithms, the tables do not show Err values. Moreover, since the number of scalarization\revv{s} is the same as the cardinality of the solution set, Card is not reported separately.  
	
	For Example 4, we do not report the results for Algorithms AUU, DLSW, and KTW since \emph{bensolve tools} was unable to perform the vertex enumeration in some iterations. For Example 5, DLSW results are not reported because of solver failures. Moreover, as the upper images are polyhedral for Example 5, the algorithms have the potential to return the exact upper image. We take $\epsilon$ sufficiently small ($10^{-8}$), hence the proximity measures are negligibly small and are not reported  for these problems.

		\begin{table}[h]
		\centering
		\caption{Computational Results for Examples 1-3 \rev{when the algorithms run until returning a finite weak $\epsilon$-solution for provided $\epsilon$ values.}\footnotemark}
		\resizebox{0.9\textwidth}{!}{
			\begin{tabular}{|c|cccc|cccc|cccc|cccc|}
				\cline{2-13}    \multicolumn{1}{r|}{} & \multicolumn{4}{c|}{Example 1, $p=3, \epsilon=0.005$} & \multicolumn{4}{c|}{Example 1, $p=4, \epsilon=0.05$} & \multicolumn{4}{c|}{Example 2, $\epsilon=0.005$} &       &       &       & \multicolumn{1}{c}{} \\
				\cline{1-13}    Algorithm & VS    & SC    & Time  & HG    & VS    & SC    & Time  & HG    & VS    & SC    & Time  & HG    &       &       &       & \multicolumn{1}{c}{} \\
				\cline{1-13}    AUU   & 572   & \textbf{124}   & 206.63 & 0.0247 & 515   & \textbf{62}    & 185.94 & 0.4366 & 387   & \textbf{39}    & 141.46 & 2.1795 &       &       &       & \multicolumn{1}{c}{} \\
				DLSW  & 1281  & 239   & 355.21 & \textbf{0.0095} & 2192  & 128   & 517.33 &  -     & 414   & 57    & 145.25 & \textbf{0.5875} &       &       &       & \multicolumn{1}{c}{} \\
				KTW   & 901   & 143   & 256.86 & 0.0141 & 1735  & 74    & 524.39 & 0.4146 & - & - & - & \multicolumn{1}{c|}{-} &       &       &       &  \multicolumn{1}{c}{} \\
				UB    & 0     & 389   & 113.15 & 0.0111 & 0     & 496   & 171.27 & \textbf{0.244} & 0     & 103   & \textbf{33.18} & 0.9108 &       &       &       & \multicolumn{1}{c}{}\\
				R     & 0     & 382.2 & \textbf{110.33} & 0.0104 & 0     & 449.8 & \textbf{145.61} & 0.2593 & 0     & 120.8 & 37.65 & 0.9166 &       &       &       & \multicolumn{1}{c}{} \\
				C     & 0     & 414   & 120.94 & 0.0118 & 0     & 485   & 154.83 & 0.2593 & 0     & 112   & 34.75 & 0.7918 &       &       &       & \multicolumn{1}{c}{} \\
				Adj   & 0     & 406   & 122.51 & 0.0117 & 0     & 462   & 147.07 & 0.2504 & 0     & 351   & 107.51 & 0.9107 &       &       &       & \multicolumn{1}{c}{} \\ \hline
				\multicolumn{1}{c|}{} & \multicolumn{4}{c|}{Example 3, $a = 5, \epsilon=0.05$} & \multicolumn{4}{c|}{Example 3, $a = 7, \epsilon=0.05$ } & \multicolumn{4}{c|}{Example 3, $a = 10, \epsilon=0.05$ } & \multicolumn{4}{c|}{Example 3, $a = 20, \epsilon=0.05$ } \\
				\hline
				Algorithm & VS    & SC    & Time  & HG    & VS    & SC    & Time  & HG    & VS    & SC    & Time  & HG    & VS    & SC    & Time  & HG \\
				\hline
				AUU   & 160   & \textbf{43}    & 62.21 & 2.0076 & 169   & \textbf{46}    & 62.7  & 2.8207 & 170   & \textbf{47}    & 64.01 & 4.1825 & 198   & \textbf{49}    & 73.11 & 8.7898 \\
				DLSW  & 1531  & 66    & 252.83 & 1.5598 & 330   & 75    & 86.36 & 1.4369 & 1060  & 72    & 168.61 & \textbf{2.299} & 323   & 77    & 85.92 & 5.0659 \\
				KTW   & 299   & 63    & 84.64 & \textbf{0.8997} & 373   & 70    & 104.72 & \textbf{1.3481} & 327$^*$   & {76}$^*$ & 93.49$^*$ & -     & 1764  & 157   & 499.27 & \textbf{1.5611} \\
				UB    & 0     & 110   & 32.51 & 1.1933 & 0     & 112   & 33.36 & 2.0808 & 0$^*$  & {90}$^*$ & \textbf{26.93$^*$}& -     & 0     & 127   & \textbf{37.44} & 7.0989 \\
				R     & 0     & 102.2 & \textbf{29.49} & 1.3571 & 0     & 112.6 & \textbf{32.41} & 3.394 & 0     & 120.8 & 34.88 & 5.0066 & 0     & 133.2 & 38.72 & 4.8624 \\
				C     & 0     & 112   & 31.8  & 1.2773 & 0     & 113   & 32.57 & 2.9109 & 0     & 126   & 36.18 & 3.9589 & 0     & 148   & 42.88 & 4.8074 \\
				Adj   & 0     & 101   & 40.26 & 1.4888 & 0     & 119   & 37.11 & 1.824 & 0$^*$     & {115}$^*$ & 40.66$^*$ & -     & 0$^*$     & {102}$^*$ & 40.05$^*$ & - \\
				\hline
			\end{tabular}%
		}
		\label{tab:ex_1-3}%
	\end{table}%
	
	\footnotetext{
		For some instances  of Example 3, the final plot of the outer approximating set was not accurate even though the approximations during the iterations seem to be correct. We indicate these results by $(\cdot)^\ast$.}

	\begin{table}[h]
		\centering
		\caption{Computational Results for Examples 4 and 5 \rev{when the algorithms run until returning a finite weak $\epsilon$-solution, where $\epsilon=0.05$ for instances of Example 4 and $\epsilon=10^{-8}$ for instances of Example 5.}}
		\resizebox{0.8\textwidth}{!}{
			\begin{tabular}{|c|ccc|ccc|ccc|ccc|}
				\cline{2-10}    \multicolumn{1}{r|}{} & \multicolumn{3}{c|}{Example 4, $a = 5$} & \multicolumn{3}{c|}{Example 4, $a = 7$} & \multicolumn{3}{c|}{Example 4, $a = 10$} &       &       & \multicolumn{1}{r}{} \\
				\cline{1-10}    Algorithm & VS & SC    & Time & VS  & SC    & Time  & VS  & SC    & Time    &       &       & \multicolumn{1}{r}{} \\
				\cline{1-10}    UB  & 0  & 1246  & 827.73 & 0 & 1295  & 909.46 & 0  & 1386  & 1019.63   &       &       & \multicolumn{1}{r}{} \\
				R  & 0   & \textbf{1085}  & \textbf{515.53} & 0 & \textbf{1137.8} & \textbf{550.34} & 0     & 1286  & 592.55   &       &       & \multicolumn{1}{r}{} \\
				C  & 0   & 1163  & 598.77 & 0 & 1203  & 579.92 & 0  & \textbf{1183}  & \textbf{529.31} &       &       & \multicolumn{1}{r}{} \\
				Adj   & 0 &  1108  & 516.48 & 0 & 1203  & 580.53 & 0   & 1290  & 603.17     &       &       & \multicolumn{1}{r}{} \\ \hline
				\multicolumn{1}{c|}{} & \multicolumn{3}{c|}{Example 5, $p=3$ (avg)} & \multicolumn{3}{c|}{Example 5, $p=4$ (ins 1)} & \multicolumn{3}{c|}{Example 5, $p=4$ (ins 2)} & \multicolumn{3}{c|}{Example 5, $p=4$ (ins 3)} \\
				\hline
				Algorithm & VS    & SC    & Time  & VS    & SC    & Time  & VS    & SC    & Time  & VS    & SC    & Time \\
				\hline
				AUU   & 602.45 & 99.6  & 196.72 & 5610  & 275   & 1672  & 667   & 80    & 233.4 & 2557  & 173   & 783.09 \\
				KTW   & 1139.5 & \textbf{96.85} & 291.07 & 14914 & \textbf{263 }  & 3763  & 1174  & \textbf{71 }   & 291.53 & 6531  & \textbf{157 }  & 1645.65 \\
				UB    & 0     & 226.65 & 61.33 & 0     & 861   & 319.7 & 0     & 161   & 47.08 & 0     & 427   & 137.02 \\
				R     & 0     & 195.2 & \textbf{51.32} & 0     & 517.6 & \textbf{177.4} & 0     & 125.8 & \textbf{32.92} & 0     & 331   & \textbf{92.86} \\
				C     & 0     & 231.55 & 63.89 & 0     & 1180  & 385.2 & 0     & 124   & 33.29 & 0     & 427   & 129.12 \\
				Adj   & 0     & 307.45 & 80.46 & 0     & 1103  & 338.9 & 0     & 147   & 41.75 & 0     & 635   & 181.58 \\
				\hline
			\end{tabular}%
		}
		\label{tab:ex_4-5}%
	\end{table}%

	When Tables \ref{tab:ex_1-3} and \ref{tab:ex_4-5} are analyzed, UB, R, C and Adj are faster than AUU, DLSW and KTW in all examples \revv{even though they solve more scalarization problems}. \revv{Note that in Examples 1-3, AUU consistently solves the least number of scalarizations whereas, in Example 5, KTW yields the smallest SC values.} \rev{Moreover, AUU has} 
	\rev{shorter CPU time\revv{s} compared to DLSW and KTW in each example.} 
	When we compare the proposed variants, we observe that R generally provides better performance in terms of runtime for the linear problems (Example 5). Moreover, for nonlinear examples with $p=4$ (\rev{Examples 1 and 4}), UB is outperformed by \rev{R, C and Adj}. For the rest of the examples, UB, R, C and Adj have similar performances in terms of runtime. There is an exception in Example 2, where Adj is outperformed by UB, R and C with a significant difference. However, Adj still has a shorter CPU time than AUU, DLSW and KTW. 
	
	\revv{When we compare the HG values in \Cref{tab:ex_1-3}, we observe that AUU yields higher HG values than the others in most cases. The best HG values are obtained by DLSW or KTW, except in Example 1, $p=4$, in which UB yields the smallest HG.}
	
	\Cref{tab:ex_1-3} does not show the results for KTW for Example 2. \rev{Indeed, KTW} can not solve this example, \rev{whose feasible region is not compact and upper image has} an asymptotic behavior. Because of this structure \rev{of the upper image}, depending on the position of $\hat{p}$, KTW finds larger distances through the iterations even though the approximations get finer. In order to illustrate this behavior, we pick $\hat{p}$ as $(1.2,1.2)^\T$ and perform \rev{a few} iterations of KTW, see \Cref{fig:TKW-ex2}. \rev{Note that $\hat{p}$ is selected as a point that is close to the boundary of $\mathcal{P}$, for illustrative purposes. This behavior of the algorithm may be observed for any $\hat{p}$, but generally after a certain number of iterations. In our computations, we did not observe this behavior illustrated in \Cref{fig:TKW-ex2} mainly because we fix $\hat{p}$ large enough. However, since $\hat{p}$ is large, it caused numerical issues.} 
	
	\begin{figure}[!]
		\centering
		\caption{\rev{The outer approximations obtained after the first, second and third iterations of the KTW algorithm for Example 2 are shown, where $\hat{p}=(1.2,1.2)^\T$. In each figure, the selected vertex for the corresponding iteration is indicated. The minimal point to be found after the scalarization is the intersection of $\bd \P$ with the dashed line between the selected vertex and $\hat{p}$. Even though the approximation gets finer in each iteration, the distances found by the KTW algorithm are 0.7619, 0.7633 and 0.7803, respectively.}}
		\begin{minipage}[b]{0.325\linewidth}
			\centering
			\includegraphics[width=\textwidth]{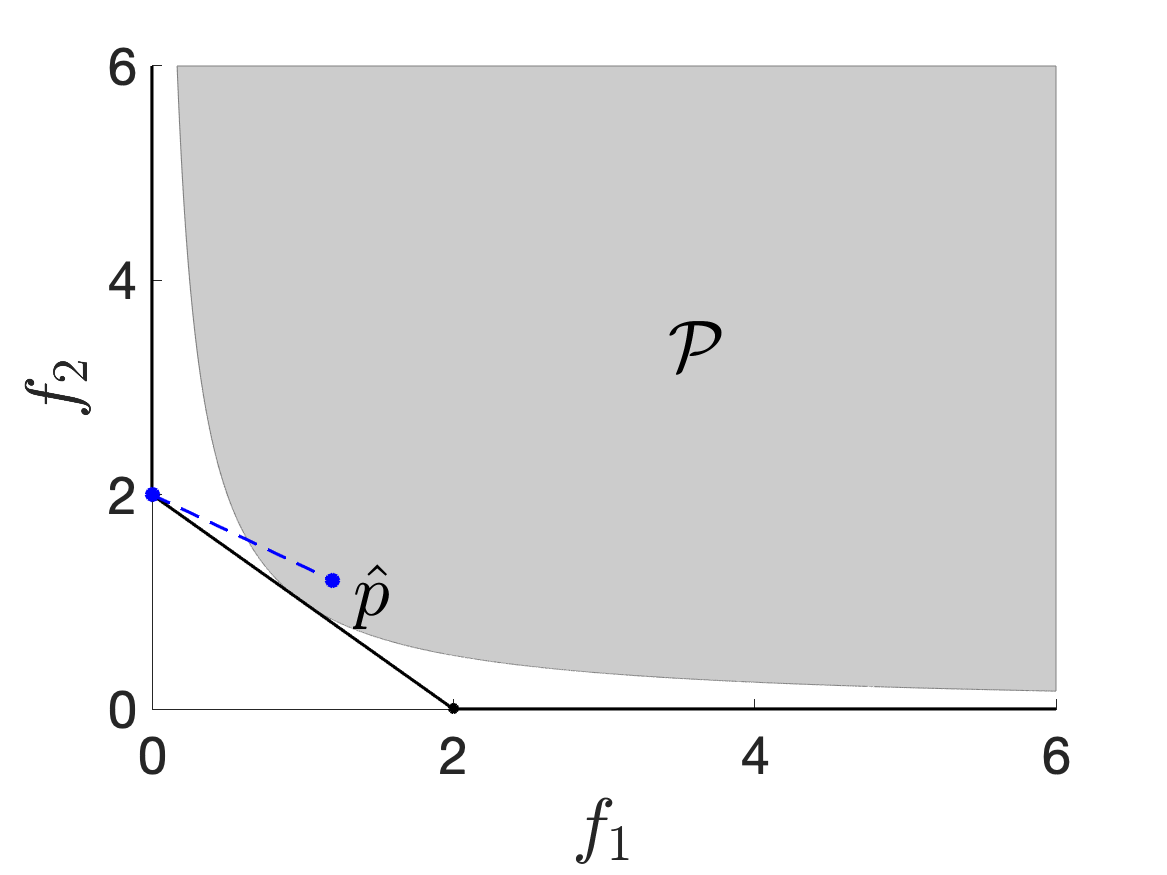}
		\end{minipage}
		\begin{minipage}[b]{0.325\linewidth}
			\centering
			\includegraphics[width=\textwidth]{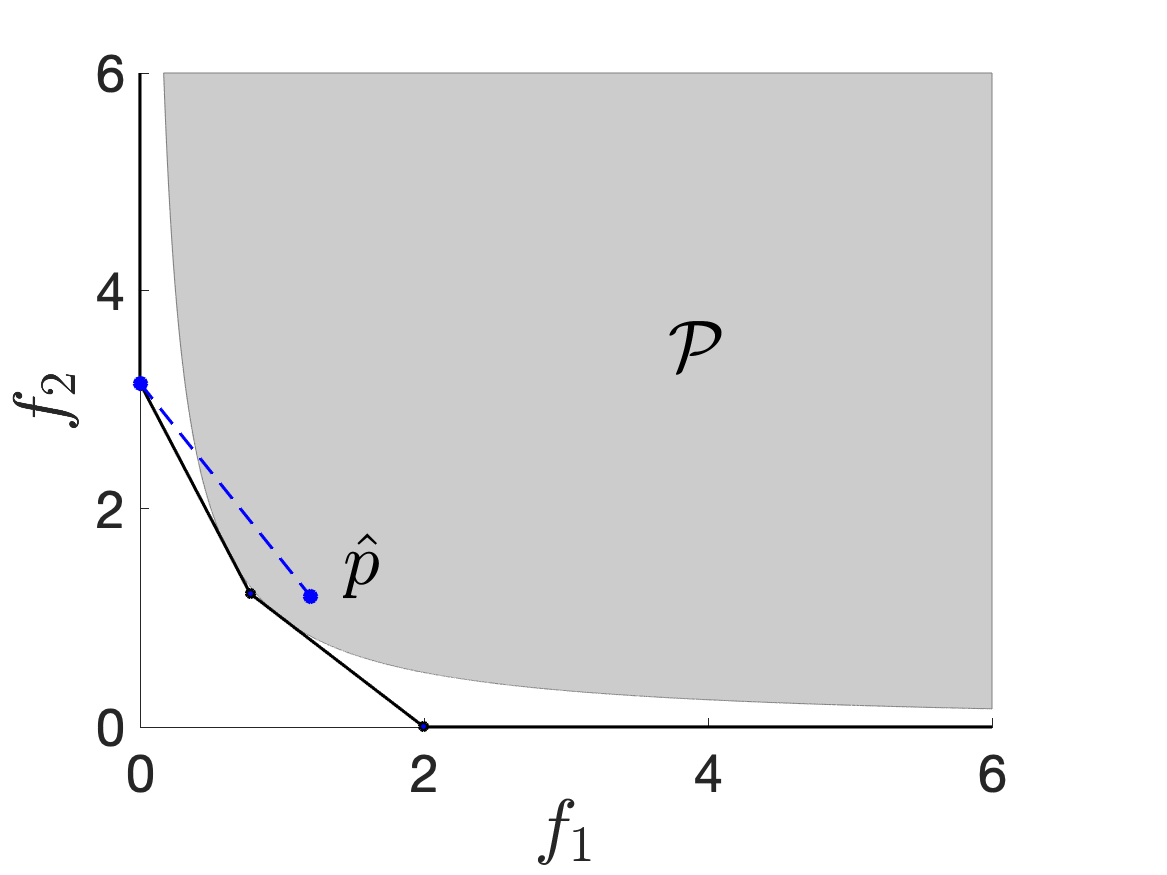}
		\end{minipage}
		\begin{minipage}[b]{0.323\linewidth}
			\centering
			\includegraphics[width=\textwidth]{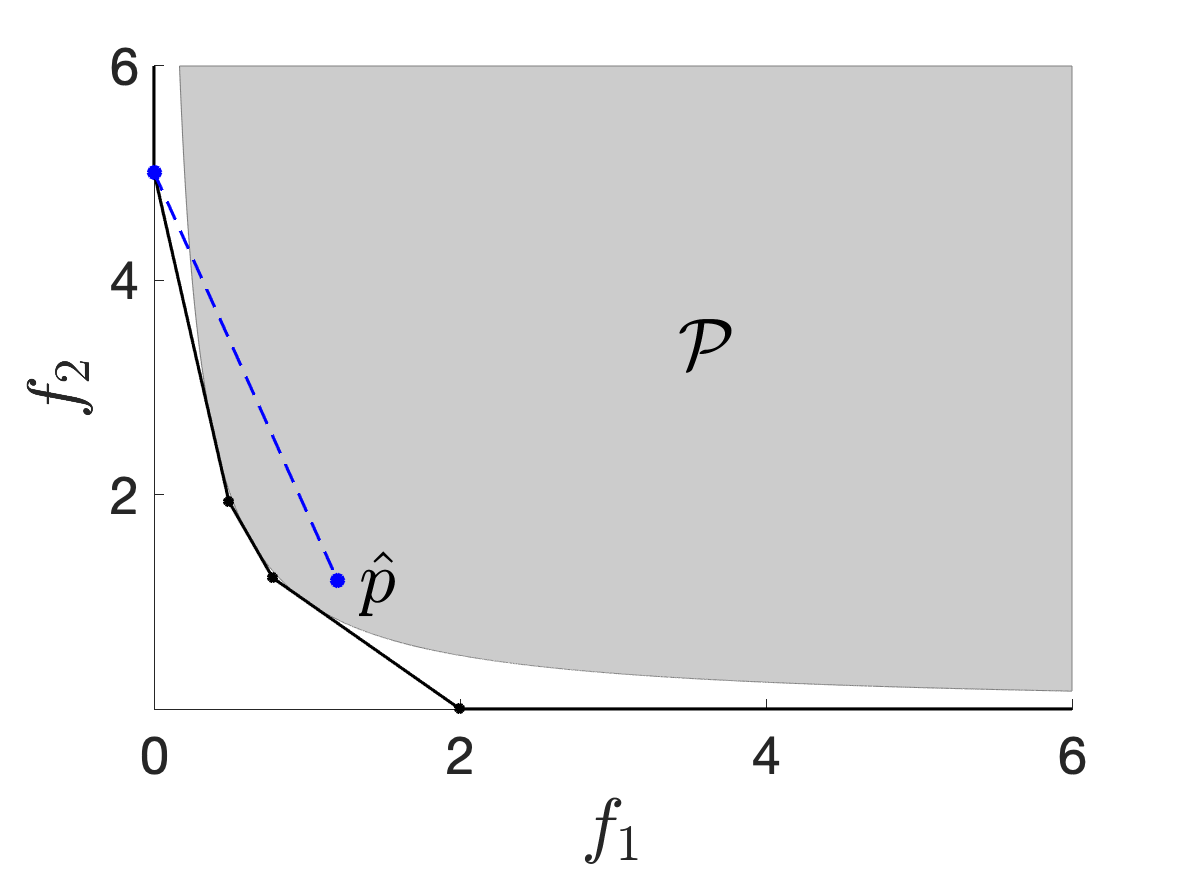}
		\end{minipage}
		\label{fig:TKW-ex2}
	\end{figure}

	\subsection{Computational results under limited runtime} \label{sect:runtime_limited}
	
\rev{In this set of experiments,} we run Examples 1-5 under \revv{a} runtime limit and \rev{report} the proximity measures that the algorithms return in \rev{\Cref{tab:run_limit}. For each example, around half of the shortest CPU time obtained in \Cref{sect:app_error} is set as the time limit.} For the variants UB, R, C and Adj, two different cardinality values are provided. The \revvv{original value is greater} (which is also equal to SC), whereas the smaller value is found by using \Cref{rem:cardinality}. 
	
	In Example 5, for $p=3$, we solve the same randomly generated instances that we \rev{solved} in \Cref{sect:app_error}. \rev{Since the} CPU times for these instances differ significantly, the most time\revv{-}consuming 8 instances are selected, and the same runtime limit of 50 seconds is fixed for them. The corresponding results in \Cref{tab:run_limit} are the averages over the 8 instances. 
	
		\begin{table}[htbp]
		\centering
		\caption{Computational results \rev{for Examples 1-5 when the algorithms run under runtime limit T.}\footnotemark}
		\resizebox{\textwidth}{!}{
			\begin{tabular}{|c|cccc|cccc|cccc|cccc|}
				\cline{2-13}    \multicolumn{1}{c|}{} & \multicolumn{4}{c|}{Example 1, $p=3$, T = 50} & \multicolumn{4}{c|}{Example 1, $p=4$, T = 75} & \multicolumn{4}{c|}{Example 2, T = 15}     &       &       &       & \multicolumn{1}{r}{} \\
				\cline{1-13}    Algorithm & VS    & Err & HG    & Card & VS    & Err & HG    & Card & VS    & Err & HG    & Card &       &       &       & \multicolumn{1}{r}{} \\
				\cline{1-13}    AUU   & 95    & 0.0168 & 0.13  & 29    & 154   & 0.0783 & -     & {28}    & 31    & 0.0296 & 32.37 & 15    &       &       &       & \multicolumn{1}{r}{} \\
				DLSW  & 139   & 0.0136 & \textbf{0.09 } & 36    & 234   & 0.0707 & 1.07  & 32    & 37    & 0.0346 & 34.54 & 15    &       &       &       & \multicolumn{1}{r}{} \\
				KTW   & 129   & 0.0183 & 0.12  & {27}    & 183   & 0.0771 & -     & 30    & - & - & - & - &       &       &       & \multicolumn{1}{r}{} \\
				UB    & 0     & 0.0057 & 0.04  & 124 (35) & 0     & 0.0348 & \textbf{0.53}  & 155 (80) & 0     & 0.0283 & \textbf{0.90}  & 44 (16) &       &       &       & \multicolumn{1}{r}{} \\
				R     & 0     & 0.0067 & 0.14  & 127.8 (44.4) & 0     & \textbf{0.0338} & 0.55  & 169.2 (98.8) & 0     & 0.0589 & 39.84 & 45 (16.2) &       &       &       & \multicolumn{1}{r}{} \\
				C     & 0     & \textbf{0.0054} & \textbf{0.09}  & 120 (32) & 0     & 0.0356 & 0.56  & 161 (95) & 0     & \textbf{0.0129} & 12.46 & 42 ({11}) &       &       &       & \multicolumn{1}{r}{} \\
				Adj   & 0     & 0.0071 & 0.13  & 124 (32) & 0     & 0.0359 & 0.66  & 166 (94) & 0     & 0.0989 & 1.79  & 46 (27) &       &       &       & \multicolumn{1}{r}{} \\ \hline
				\multicolumn{1}{r|}{} & \multicolumn{4}{c|}{Example 3, $a=5$, T = 15} & \multicolumn{4}{c|}{Example 3, $a=7$, T = 15} & \multicolumn{4}{c|}{Example 3, $a=10$, T = 15} & \multicolumn{4}{c|}{Example 3, $a=20$, T = 15} \\
				\cline{2-17}    \multicolumn{1}{c|}{} & VS    & Err & HG    & Card & VS    & Err & HG    & Card & VS    & Err & HG    & Card & VS    & Err & HG    & Card \\ \hline
				AUU   & 25    & 0.4119 & 18.37 & 10    & 28    & 0.4138 & 25.64 & 11    & 28    & 0.4033 & 38.89 & 11    & 25    & 0.6616 & 100.51 & 11 \\
				DLSW  & 42    & 0.4625 & 23.34 & 9     & 33    & 0.6168 & 19.73 & 12    & 42    & 0.4424 & -     & 12    & 35    & 0.2958 & 59.14 & 13 \\
				KTW   & 35    & 0.2274 & 8.48  & 12    & 34    & 0.1839 & 5.83  & 16    & 37    & 0.1953 & 10.32 & 16    & 36    & 0.3164 & 25.38 & 16 \\
				UB    & 0     & 0.1744 & 2.76  & 34 (14) & 0     & 0.2163 & \textbf{4.09}  & 35 ({9}) & 0     & 0.4800 & -     & 36 (9) & 0     & 0.3712 & 12.00 & 37 ({6}) \\
				R     & 0     & 0.5984 & 5.21  & 35.8 (10) & 0     & 0.2862 & 8.38  & 36.4 (9.2) & 0     & 0.8142 & 18.76 & 37 (11.2) & 0     & 0.4522 & 33.37 & 37.4 (11.8) \\
				C     & 0     & \textbf{0.1057} &\textbf{ 2.54}  & 35 ({8}) & 0     & \textbf{0.1494} & 4.99  & 36 (10) & 0     & \textbf{0.1672} & \textbf{5.38 } & 37 (\textbf{7}) & 0     & \textbf{0.1878} & 12.45 & 37 (10) \\
				Adj   & 0     & 0.2059 & 3.08  & 31 (10) & 0     & 0.4557 & 6.11  & 35 (14) & 0     & 0.1850 & 5.91  & 34 (11) & 0     & 0.5074 & \textbf{11.54} & 37 (9) \\
				\hline
				\multicolumn{1}{r|}{} & \multicolumn{4}{c|}{Example 4, $a=5$, T = 250} & \multicolumn{4}{c|}{Example 4, $a=7$, T = 250} & \multicolumn{4}{c|}{Example 4, $a=10$, T = 250} &       &       &       & \multicolumn{1}{r}{} \\ 
				\cline{2-13}    \multicolumn{1}{r|}{} & VS    & Err & HG    & Card & VS    & Err & HG    & Card & VS    & Err & HG    & Card &       &       &       & \multicolumn{1}{r}{} \\ 
				\cline{1-13}    UB    & 0     & \textbf{0.1010} & -     & 392 (233) & 0     & 0.1261 & 7.23  & 366 ({211}) & 0     & \textbf{0.1010} & -     & 382 ({214}) &       &       &       & \multicolumn{1}{r}{} \\
				R     & 0     & 0.2072 & -     & 508.5 ({164.5}) & 0     & 0.1825 & -     & 480.4 (318.4) & 0     & 0.2759 & -     & 498.6 (329) &       &       &       & \multicolumn{1}{r}{} \\
				C     & 0     & 0.1678 & -     & 452 (304) & 0     & \textbf{0.1013} & -     & 424 (276) & 0     & 0.3343 & -     & 519 (361) &       &       &       & \multicolumn{1}{r}{} \\
				Adj   & 0     & 0.2072 & -     & 537 (378) & 0     & 0.1312 & -     & 504 (349) & 0     & 0.1518 & -     & 513 (351) &       &       &       & \multicolumn{1}{r}{} \\
				\hline
				\multicolumn{1}{c|}{} & \multicolumn{4}{c|}{Example 5, $p=3$, T = 50 (avg)} & \multicolumn{4}{c|}{Example 5, $p=4$, T = 100 (ins 1)} & \multicolumn{4}{c|}{Example 5, $p=4$, T= 15 (ins 2)} & \multicolumn{4}{c|}{Example 5, $p=4$, T = 50 (ins 3)} \\
				\cline{2-17}    \multicolumn{1}{c|}{} & VS    & Err & HG    & Card & VS    & Err & HG    & Card & VS    & Err & HG    & Card & VS    & Err & HG    & Card \\ \hline
				AUU   & 69.125 & 1.3145 & 15966.61 & 22.13 & 318   & 2.7899 & 3453462.73
				& 38    & 43    & 11.3862 & 415587.08
				& {12}    & 182   & 2.9557 & 722544.84
				& 26 \\
				DLSW  & 72.375 & 2.1944 & 22167.71 & {20.25} & 338   & 3.1038 &    2776469.82
				& 33    & 41    & 5.6440 & 224514.46
				& 13    & 175   & 3.2645 & 
				803350.19
				& {21} \\
				KTW   & 98    & 0.9805 & 11919.47 & 26.25 & 403   & 4.8830 & 
				3176108.34
				& {30}    & 57    & 10.2325 & 182864.48
				& 13    & 194   & 1.5588 & 408030.61
				& 30 \\
				UB    & 0     & 1.1396 & 32596.53 & 95.38 (25.63) & 0     & \textbf{1.2209} & 
				-
				& 269 (101) & 0     & 3.0480 & -
				& 56 (20) & 0     & 1.0328 & 
				\textbf{37740.57}
				& 165 (66) \\
				R     & 0     & 2.8557 & 3564.33 & 91.75 (22.13) & 0     & 2.5621 & -     & 286.4 (82.6) & 0     & 8.6336 & 
				23599.91 & 55.4 (14.4) & 0     & 3.1134 &     42125.39
				& 163.8 (44) \\
				C     & 0     & \textbf{0.9553} & \textbf{2825.69} & 89 (21) & 0     & 1.4384 & \textbf{139495.18}
				& 266 (111) & 0     & 1.7226 & 
				\textbf{13700.71} & 55 (18) & 0     & \textbf{0.8621} & - & 150 (50) \\
				Adj   & 0     & 1.0659 & 3209.04 & 92.5 (29.63) & 0     & 1.8833 &     268168.59  & 336 (182) & 0     & \textbf{0.8027} & 19003.47 & 57 (20) & 0     & 1.1001 &  46920.04 & 180 (88) \\
				\hline
			\end{tabular}%
		}
		\label{tab:run_limit}%
	\end{table}%
	\footnotetext{$\epsilon$ is taken as in Tables \ref{tab:ex_1-3} and \ref{tab:ex_4-5} even though it is not the stopping criteria.}

	From \Cref{tab:run_limit}, we observe that the solution sets returned by the variants (when \Cref{rem:cardinality} is applied for UB, R, C and Adj) have similar cardinality values for two and three-dimensional examples. In Examples 2 and 3, C returned smaller approximation error\revv{s} compared to the other variants, in general. In Example 1 with $p=3$, proposed variants return smaller error values compared to the variants from the literature, but in Example 5 with $p=3$, there is no clear conclusion. 
	
	For four-dimensional problems, AUU, DLSW, and KTW returned solution sets with smaller cardinality values than the others. In return, the approximation errors returned by AUU, DLSW, and KTW are larger than UB, C, and Adj for these problems. On the other hand, R behaves similar\revv{ly} to AUU, DLSW, and KTW. \revv{In} most of the cases, HG values behave similar\revv{ly} to Err. However, in Example 5 for problems with $p=3$ and (ins 3) of $p=4$, the proposed variants return smaller HG, even though Err values are comparable among the algorithms.
	
	An additional computational study is conducted on Example 3 to observe the behavior of the algorithms as the time limit changes. The variants are terminated after 50, 100 and 200 seconds. In this set of experiments, we also observe the effect of $\epsilon$ used within the variants UB, C, Adj and R even if the $\epsilon$ \rev{does} not \rev{determine} the stopping condition. Since we do not select the farthest away vertex \rev{to proceed with}, $\epsilon$ has still an important role to decide if the algorithm would add a cut to the current outer approximation or not. To observe this effect, we run all the variants for the above\revv{-}mentioned time limits both with $\epsilon=0.005$ and $\epsilon=0$. {The results are provided \revv{in the appendix} and the approximation errors obtained within different time limits for $a=5$ are shown in \Cref{fig:error_time}.}\footnote{The results for other $a$ values are similar, hence the corresponding figures are not included here.}
		
	\begin{figure}[h]		
		\caption{Approximation error results under runtime limits \rev{$T\in\{50,100,200\}$} for Example 3, $a=5$ when $\epsilon$ is taken as 0 (left) and 0.005 (right).}
		\centering  
		\subfloat{
			\resizebox{0.493\textwidth}{!}{%
				
				\begin{tikzpicture}
				\begin{axis}[
				ylabel={Approximation Error},
				xmin=0, xmax=250,
				ymin=0, ymax=0.18,
				ytick={0, 0.02, 0.04, 0.06, 0.08, 0.1, 0.12, 0.14, 0.16},
				yticklabels={0, 0.02, 0.04, 0.06, 0.08, 0.1, 0.12, 0.14, 0.16},
				xtick={0,50,100,200},
				xticklabels={0, T=50,T=100,T=200},
				legend style={at={(0.5,-0.1)},anchor=north},
				ymajorgrids=true,
				grid style=dashed,
				legend columns=5, scaled y ticks = false
				]

				\addplot[ line width=1pt,mark size=2pt,
				color=green,
				mark=square*,
				]
				coordinates {
					(50,0.061)(100,0.0293)(200, 0.016)
				};
				\addlegendentry{AUU};
				
				\addplot[ line width=1pt,mark size=2pt,
				color=magenta,
				mark=triangle*,
				]
				coordinates {
					(50,0.1428)(100,0.098)(200,0.0546)
				};
				\addlegendentry{DLSW};

				\addplot[line width=1pt,mark size=2.3pt,
				color=black,
				mark=halfsquare*,
				]
				coordinates {
					(50,0.0633)(100,0.0355)(200,0.0213)
				};
				\addlegendentry{KTW};
				
				\addplot[line width=1pt,mark size=2pt,
				color=blue,
				mark=diamond*,
				]
				coordinates {
					(50,0.0925)(100,0.0448)(200,0.0448)
				};
				\addlegendentry{C};
				
				\addplot[line width=1pt,mark size=2.3pt,
				color=red,
				mark=10-pointed star,
				]
				coordinates {
					(50,0.0775)(100,0.0422)(200,0.0405)
				};
				\addlegendentry{Adj};
				
				\end{axis}
				\end{tikzpicture}
		}}
		\hfill
		\subfloat{
			\resizebox{0.48\textwidth}{!}{%
				\begin{tikzpicture}
				\begin{axis}[
				ylabel={Approximation Error},
				xmin=0, xmax=250,
				ymin=0, ymax=0.09,
				ytick={0, 0.01, 0.02, 0.03, 0.04, 0.05, 0.06, 0.07, 0.08},
				yticklabels={0, 0.01, 0.02, 0.03, 0.04, 0.05, 0.06, 0.07, 0.08},
				xtick={0,50, 100, 200},
				xticklabels={0, T=50,T=100,T=200},
				legend style={at={(0.5,-0.1)},anchor=north},
				ymajorgrids=true,
				grid style=dashed,
				legend columns=5,
				scaled y ticks = false
				]

				\addplot[line width=1pt,mark size=2pt,
				color=green,
				mark=square*,
				]
				coordinates {
					(50,0.061)(100,0.0319)(200, 0.0151)
				};
				\addlegendentry{AUU};

				\addplot[line width=1pt,mark size=2.3pt,
				color=black,
				mark=halfsquare*,
				]
				coordinates {
					(50,0.0633)(100,0.0355)(200,0.0213)
				};
				\addlegendentry{KTW};
				
				\addplot[line width=1pt,mark size=2pt,
				color=cyan,
				mark=halfcircle*,
				]
				coordinates {
					(50,0.0489)(100,0.0489)(200,0.0108)
				};
				\addlegendentry{UB};

				\addplot[line width=1pt,mark size=2pt,
				color=blue,
				mark=diamond*,
				]
				coordinates {
					(50,0.0472)(100,0.0447)(200,0.0228)
				};
				\addlegendentry{C};
				
				\addplot[line width=1pt,mark size=2.3pt,
				color=red,
				mark=10-pointed star,
				]
				coordinates {
					(50,0.0709)(100,0.0329)(200,0.0178)
				};
				\addlegendentry{Adj};
				
				\end{axis}
				\end{tikzpicture}
		}}
		\label{fig:error_time}		
	\end{figure}
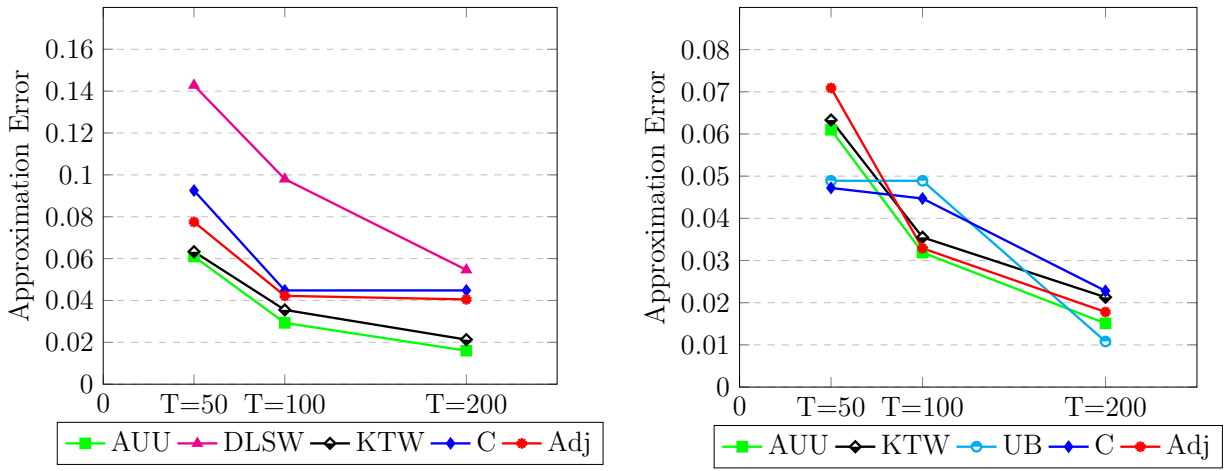
	
	As expected, the performance of AUU, KTW and DLSW are not affected by the change in $\epsilon$. On the other hand, an increase in $\epsilon$ affects the performance of UB, C and Adj positively. R is affected less compared to UB, C and Adj, while UB is affected significantly. Note that the two variant\revv{s} that give the worst approximation error in each setting (UB and R for $\epsilon =0$ and R and DLSW for $\epsilon = 0.005$) are not plotted \revv{to increase visibility.}


	\subsection{Computational results under fixed cardinality} \label{sect:Fixed Cardinality}
	
	In this section, we compare the performances of the algorithms when they run until finding a solution set with a predetermined cardinality. 
	For each example, we give a cardinality limit so that the algorithms terminate before reaching the approximation error that is set for the experiments in \Cref{sect:app_error}. In Example 5 where $p=3$, the same 8 instances are selected as in \Cref{sect:runtime_limited}. For the variants UB, R, C, and Adj, we apply \Cref{rem:cardinality}. \rev{The results can be seen in \Cref{tab:fixedcard}.}
	
		\begin{table}[h]
		\centering
		\caption{Computational results \rev{for Examples 1-5 when the algorithms run until finding a solution set with predetermined cardinality (Card).}}
		\resizebox{\textwidth}{!}{
			\begin{tabular}{|c|cccc|cccc|cccc|cccc|}
				
				\cline{2-13}    \multicolumn{1}{c|}{} & \multicolumn{4}{c|}{Example 1, $p=3$, Card $= 100$} & \multicolumn{4}{c|}{Example 1, $p=4$, Card $= 50$} & \multicolumn{4}{c|}{Example 2, Card $=20$ } &       &       &       & \multicolumn{1}{c}{} \\
				\cline{1-13}    \multicolumn{1}{|c|}{Algorithm}  & VS    & T  & Err & HG    & VS    & T  & Err & HG    & VS    & T  & Err & HG    &       &       &       & \multicolumn{1}{c}{} \\ \cline{1-13} 
				AUU   & 451   & 252.12 & \textbf{0.0062} & 0.03  & 401   & 197.53 & \textbf{0.0578} & 0.51  & 64    & 29.15 & 0.0197 & 14.98 &       &       &       & \multicolumn{1}{c}{} \\
				DLSW  & 455   & 175.98 & 0.0065 & 0.03  & 430   & 144.85 & 0.0579 & 0.52  & 90    & 34.97 & 0.0188 & 16.06 &       &       &       & \multicolumn{1}{c}{} \\
				KTW   & 695   & 266.51 & 0.0067 & 0.03  & 816   & 327.38 & 0.0614 & -     & -     & -     & -     & -     &       &       &       & \multicolumn{1}{c}{} \\
				UB    & 0     & 93.74 & 0.0196 & 0.03  & 0     & 62.20 & 0.1448 & -     & 0     & 15.83 & 0.0283 & \textbf{11.77} &       &       &       & \multicolumn{1}{c}{} \\
				R     & 0     & \textbf{51.98} & 0.0325 & 0.07  & 0     & \textbf{48.24} & 0.4142 & 1.13  & 0     & 16.81 & 0.2236 & 32.63 &       &       &       & \multicolumn{1}{c}{} \\
				C     & 0     & 90.66 & 0.0195 & 0.07  & 0     & 54.12 & 0.0894 & \textbf{0.49}  & 0     & 21.45 & \textbf{0.0129} & 81.53 &       &       &       & \multicolumn{1}{c}{} \\
				Adj   & 0     & 84.96 & 0.0116 & 0.04  & 0     & 48.75 & 0.1281 & 1.02  & 0     & \textbf{11.72} & 0.0989 & -     &       &       &       & \multicolumn{1}{c}{} \\ \hline
				\multicolumn{1}{c|}{} & \multicolumn{4}{c|}{Example 3, $a = 5$, Card $= 30$} & \multicolumn{4}{c|}{Example 3, $a = 7$, Card $= 30$} & \multicolumn{4}{c|}{Example 3, $a = 10$, Card $= 30$} & \multicolumn{4}{c|}{Example 3, $a = 20$, Card $= 30$} \\ \cline{2-17}
				\multicolumn{1}{c|}{} & VS    & T  & Err & HG    & VS    & T  & Err & HG    & VS    & T  & Err & HG    & VS    & T  & Err &  HG \\ \hline
				AUU   & 98    & 50.51 & 0.0789 & 3.88  & 100   & 50.48 & 0.0789 & 4.32  & 104   & 52.06 & 0.0874 & 8.37  & 123   & 59.48 & \textbf{0.0877} & 15.77 \\
				DLSW  & 581   & 174.80 & 0.0979 & 4.22  & 106   & 41.22 & 0.1112 & 4.86  & 505   & 150.75 & 0.1122 & 7.70  & 104   & 41.40 & 0.1264 & 15.61 \\
				KTW   & 106   & 106.15 & \textbf{0.0543} & \textbf{1.96}  & 102   & 78.82 & \textbf{0.0682} & \textbf{2.83}  & 97    & 71.16 & \textbf{0.0639} & \textbf{4.07}  & 116   & 76.91 & 0.0945 & 10.21 \\
				UB    & 0     & 25.14 & 0.0629 & 2.94  & 0     & 25.78 & 0.2163 & 6.36  & 0     & 26.66 & 0.1056 & 6.43  & 0     & 26.85 & 0.1256 & \textbf{7.29} \\
				R     & 0     & 25.27 & 0.1752 & 6.88  & 0     & 24.84 & 0.1677 & 8.97  & 0     & 25.59 & 0.2086 & 13.38 & 0     & 26.26 & 0.3129 & 15.93 \\
				C     & 0     & 24.32 & 0.1006 & 3.63  & 0     & 25.87 & 0.1184 & 5.71  & 0     & 26.10 & 0.1105 & 6.78  & 0     & 25.18 & 0.1827 & 7.98 \\
				Adj   & 0     & \textbf{23.34} & 0.1020 & 4.39  & 0     & \textbf{21.31} & 0.4557 & 7.82  & 0     & \textbf{24.69} & 0.1535 & 11.55 & 0     & \textbf{24.97} & 0.5069 & 7.84 \\
				\hline
				\multicolumn{1}{c|}{} & \multicolumn{4}{c|}{Example 4, $a = 5$, Card $= 100$} & \multicolumn{4}{c|}{Example 4, $a = 7$,  Card $= 100$} & \multicolumn{4}{c|}{Example 4,  $a = 10$, Card $= 100$} &       &       &       & \multicolumn{1}{c}{} \\ \cline{2-13}
				\multicolumn{1}{c|}{} & VS    & T  & Err & HG    & VS    & T  & Err & HG    & VS    & T  & Err & HG    &       &       &       & \multicolumn{1}{c}{} \\ \cline{1-13}  
				UB    & 0     & 137.56 & 0.1221 & -     & 0     & 131.28 & \textbf{0.1261} & -     & 0     & 145.59 & \textbf{0.2072} & -     &       &       &       & \multicolumn{1}{c}{} \\
				R     & 0     & 112.89 & 0.4142 & 16.22 & 0     & 113.22 & 0.3092 & -     & 0     & 115.72 & 0.4101 & -     &       &       &       & \multicolumn{1}{c}{} \\
				C     & 0     & 142.94 & \textbf{0.1039} & 9.28  & 0     & 148.59 & 0.1530 & 17.54 & 0     & 129.40 & 0.3343 & -     &       &       &       & \multicolumn{1}{c}{} \\
				Adj   & 0     & \textbf{100.93} & 0.175 & \textbf{7.48}  & 0     & \textbf{99.75} & 0.2072 & -     & 0     & \textbf{113.33} & 0.2226 & -     &       &       &       & \multicolumn{1}{c}{} \\ \hline
				\multicolumn{1}{c}{} & \multicolumn{4}{|c}{Example 5, $p = 3$, Card $= 50$ (avg)} & \multicolumn{4}{|c}{Example 5, $p =4$, Card $= 125$ (ins 1)} & \multicolumn{4}{|c|}{Example 5, $p =4$, Card $= 40$ (ins 2)} & \multicolumn{4}{c|}{Example 5, $p =4$, Card $= 75$ (ins 3)} \\    \cline{2-17}
				\multicolumn{1}{c|}{} & VS    & T  & Err & HG    & VS    & T  & Err & HG    & VS    & T  & Err & HG    & VS    & T  & Err & HG \\ \hline
				AUU   & 204   & 72.53 & \textbf{0.2502} & 5573.2 & 2533  & 768.70 & \textbf{0.229} & 606958.2
				& 266   & 87.09 & 0.4871 & 30354.2
				& 1207  & 386.52 & 0.4436 & 121479.4
				\\
				DLSW  & 233   & 81.98 & 0.4715 & 3567.4 & 3042  & 904.62 & 0.348 & 
				163459.1
				& 298   & 95.81 & 0.3892 & 12517.2
				& 2027  & 363.61 & 0.5809 & 
				79236.2 \\
				KTW   & 347.75 & 94.26 & 0.2564 & 4484.0 & 6457  & 1641.78 & 0.256 & 
				605558.1
				& 609   & 156.71 & 0.2523 & 24715.2
				& 1167  & 534.57 & \textbf{0.3119} & 				
				138802.3
				\\
				UB    & 0     & 41.84 & 0.3672 & \textbf{1025.2} & 0     & 110.63 & 1.221 & 
				163636.9
				& 0     & 25.16 & 0.4871 & 4128.8
				& 0     & 58.71 & 1.0328 & 
				27890.2
				\\
				R     & 0     & 43.31 & 1.0093 & 1325.7 & 0     & 127.41 & 1.206 & 
				\textbf{40282.1}
				& 0     & 26.85 & 0.3716 &     1112.3 & 0     & 65.57 & 0.9916 & 
				\textbf{10856.6}
				\\
				C     & 0     & 46.17 & 0.4707 & 1051.4 & 0     & 106.68 & 0.450 & 
				118499.2
				& 0     & 26.84 & 0.2489 & \textbf{876.6} & 0     & 63.51 & 0.8274 & 
				89151.3
				\\
				Adj   & 0     & \textbf{36.08} & 0.6722 & 1860.9 & 0     & \textbf{79.83} & 1.883 & 
				331493.0 & 0     & \textbf{24.36} & \textbf{0.1966} & 1636.87 & 0     & \textbf{47.58} & 1.1001 & 
				54859.5 \\  \hline
			\end{tabular}%
		}
		\label{tab:fixedcard}%
	\end{table}%

	In line with the observations from \Cref{sect:app_error}, we observe from \Cref{tab:fixedcard} that UB, R, C, and Adj require less CPU time compared to AUU, DLSW, and KTW. Among the algorithms from \Cref{sect:alg_literature}, DLSW requires less CPU time compared to AUU and KTW, especially if the dimension of the objective space is high. When we compare the proposed variants, we see that Adj is slightly faster than the others in most examples. 
	
	\rev{The} approximation errors \rev{obtained} by AUU, DLSW, and KTW are smaller than or very close to those \rev{obtained} by UB, R, C, and Adj, in general. This shows the trade-off between the runtime and approximation error. When we compare the algorithms from Section \ref{sect:alg_literature}, AUU and KTW yield better results in terms of the approximation error compared to DLSW, in general. On the other hand, the variants UB, R, C, and Adj are comparable \rev{among themselves}. When we consider the HG values, we observe a similar behavior as Err, especially for problems with $p\in\{2,3\}$. However, in Example 5 the proposed variants return smaller HG values even though the Err values are slightly worse than or comparable with the algorithms from the literature. 
	

	\section{Conclusion}\label{ch:conclusion}
	
	We \rev{consider bounded convex vector optimization problems and} present a general framework \rev{for existing outer approximation algorithms from the literature. We observe that many algorithms iterate by solving Pascoletti-Serafini scalarizations or equivalent models and they mainly differ in selecting the parameters of the scalarization in each iteration.} We propose \rev{additional} methods to select \rev{the reference point (a vertex of the current outer approximation) and direction parameter} of this scalarization. 
	
 	\rev{First, for a given vertex, we propose two methods to select the direction parameter. These methods are based on the position of the selected vertex as well as (1) the ideal point and (2) the adjacent vertices of the selected vertex, respectively.  We} compare \rev{the proposed} direction selection rules \rev{together with the ones commonly used in the literature through preliminary computational tests} and observe that \rev{using the positions of the adjacent vertices yields} promising results. \revv{Hence, if the vertex enumeration problems are solved based on the double description method (e.g. by bensolve) so that the adjacency information is already obtained, then using Adj to select the direction parameter for each vertex is preferable. }
 	
 	We also propose some vertex selection rules which, different \revv{from} the existing \rev{ones} \revv{in} the literature, do not require solving additional optimization problems. \rev{Instead, the proposed selection methods use the positions of all vertices (1) to form clusters and select accordingly, (2) to pick the most isolated vertex, or (3) to find the farthest away vertex from its corresponding local upper bound, respectively.}
	
	We implement the proposed \revv{vertex selection procedures together with the direction selection rule Adj} and three relevant approaches from the literature\revv{. We provide an extensive computational study and} observe that \rev{especially when the ordering cone is the positive orthant or a larger cone, the proposed variants} perform better in terms of CPU time if the stopping condition is the approximation error or the cardinality \rev{of the solution set. On the other hand, the CPU times are comparable when the ordering cone is a strict subset of the positive orthant}.\footnote{\revvv{See the appendix for the results with $C \neq \R^p_+$.}} 
	
	\rev{When the algorithms are run under a time limit, the proximity measures returned by the proposed variants are either better than or comparable to the ones returned by the algorithms from the literature.} 
	We also observe that the selection of $\epsilon$ affects the performance of the proposed variants even when $\epsilon$ is not used as the stopping criterion.
		
	\revv{We conclude that if a reasonably good direction selection method (e.g. Adj) is used in the algorithm, then the vertex selection rule may have less impact on the performance of the algorithm. Hence, solving additional single objective optimization problems to select a vertex may not be necessary to improve the performance of the algorithm.}
	
	\revv{When we compare the overall performances of the proposed vertex selection methods, there is no clear conclusion. Indeed, random selection works very well in many cases, especially if the stopping condition is the approximation error. However, there are also some cases for which R works considerably worse than the proposed variants{, see for instance the results for Examples 3 and 5 when the algorithms run until a predetermined runtime or cardinality limit (Tables \ref{tab:run_limit}, \ref{tab:fixedcard})}. In that sense, we conclude that R may not be as robust as the proposed variants, as expected. As a final remark, note that if the ordering cone is the positive orthant, then using UB has an advantage as it also returns an upper bound for the current approximation error when it is terminated after a predetermined runtime or cardinality limit.}
	

	\section*{Acknowledgments}
	
	The authors thank the anonymous referees for insightful comments that allowed them to correct some inaccuracies appearing in the preceding version and for numerous suggestions that improved the presentation. They also thank Daniel D\"orfler for his suggestions to improve the computational results.	
		
	\bibliographystyle{plain}	
	\bibliography{references}
	
	\appendix

	\section{Testing different direction selection rules for UB} ~\\
Even though we fix Adj as the direction selection rule for further computational study in \revv{Section 5.2}, we note that when UB is used as the vertex selection rule, it is possible to use the local upper bounds in order to determine the direction parameter \rev{as well}. In particular, for a vertex $v$ with the corresponding local upper bound $u$, we consider selecting the direction parameter as the vector from $v$ to $u$. In order to test this direction selection rule we conduct a computational study on Example 3 in which we compare three direction selection rules: $d=\frac{e}{\norm{e}}$ (UB-FD), $d=\frac{(u-v)}{\norm{u-v}}$ (UB-UB), and adjacent vertices approach (UB-Adj). We observe in \Cref{table:pre_ub} that UB-Adj solves less number of scalarizations while UB-UB clearly has a negative effect on T and SC. Based on this preliminary analysis, we \rev{confirm the usage of} Adj as the direction selection rule throughout the computational studies in \revv{Section 6}.

\begin{table}[h]
	\centering
	\caption{The effect of different direction selection methods for UB for Ex 3 with $\epsilon=0.05$}
	\resizebox{\textwidth}{!}{
		\begin{tabular}{|c|ccc|ccc|ccc|ccc|} 
			\cline{2-13} \multicolumn{1}{c}{} 
			& \multicolumn{3}{|c|}{$a=5$} & \multicolumn{3}{c|}{$a=7$} & \multicolumn{3}{c|}{$a=10$} &\multicolumn{3}{c|}{$a=20$}  \\ 
			\multicolumn{1}{c|}{}  & UB-FD & UB-UB & UB-Adj & UB-FD & UB-UB & UB-Adj & UB-FD & UB-UB & UB-Adj & UB-FD & UB-UB & UB-Adj\\
			\hline
			SC & 128  & 547 & \textbf{110 }& 127 & 570 & \textbf{112} & 139 & 737 &\textbf{ 90} & 194 & 1990 & \textbf{127} \\ 
			T & 38.83  & 173.58  & \textbf{32.51} & 36.81 & 178.73 & \textbf{33.36 }& 41.20 & 240.06 & \textbf{26.93} & 62.16 & 996.51 & \textbf{37.44} \\ 
			\hline
	\end{tabular}}\label{table:pre_ub}
\end{table}

\section{Vertex selection procedures of the existing algorithms from Section 4.4}
\begin{align} \label{eq:GB2}
\textrm{maximize } \lambda \:\: \textrm{subject to }  f(x) \leq \hat{p}+\lambda(v-\hat{p}), \: x \in \X, \: \lambda \in \R.
\end{align}	
\begin{procedure}[H]
	\caption{SelectVertex($V^k$,$\Vused$,$\Vinfo$) of Algorithm KTW}
	\begin{algorithmic}[1] 
		\STATE $ V_{\textnormal{info}}^{\textnormal{temp}} =\emptyset$
		\FORALL {$v \in V^k \setminus \Vused$}	
		\IF{$(v,y^v, z^v) \notin \Vinfo$}
		\STATE Let ($x^v,\lambda^v$) be a solution to  \eqref{eq:GB2}, $y^v:=\hat{p}+\lambda^v(v-\hat{p})$ and $z^v:=\norm{y^v-v}$;
		\ENDIF
		\STATE $V_{\textnormal{info}}^{\textnormal{temp}} \gets V_{\textnormal{info}}^{\textnormal{temp}}  \cup \{(v,y^v,z^v)\}$;
		\ENDFOR
		\STATE $\Vinfo \gets V_{\textnormal{info}}^{\textnormal{temp}}  $;
		\STATE Let $v^* \in\argmax_{(v,y^v,z^v) \in \mathcal{V}^{k}} z^v$;
		\RETURN $v^*$,  $\Vinfo$
	\end{algorithmic}\label{proc:KTW}
\end{procedure}

\begin{align} \label{eq:QP}
\textrm{minimize } \norm{y-v}^2 \:\: \textrm{subject to } y \in \mathcal{I}^k.
\end{align}	
\begin{procedure}[H]
	\caption{SelectVertex($x^k$,$V^k$,$\Vused$,$\Vinfo$) of Algorithm DLSW}
	\begin{algorithmic}[1] 
		\STATE $ V_{\textnormal{info}}^{\textnormal{temp}} =\emptyset$
		\FORALL {$v \in V^k \setminus \Vused$}	
		\IF{ $(v,y^v,z^v)\in \Vinfo$}
		\IF{$(y^v-v)^\T (f(x^{k})-y^v) < 0$}
		\STATE Solve \eqref{eq:QP} and for the vertex $v$ find  $y^v \in \mathcal{I}^k$;
		\ENDIF
		\ELSE
		\STATE Solve \eqref{eq:QP} and for the vertex $v$ find  $y^v \in \mathcal{I}^k$;
		\ENDIF
		\STATE $V_{\textnormal{info}}^{\textnormal{temp}}  \gets V_{\textnormal{info}}^{\textnormal{temp}} \cup \{(v,y^v,z^v)\}$;
		\ENDFOR
		\STATE $\Vinfo \gets V_{\textnormal{info}}^{\textnormal{temp}} $;
		\STATE Let $v^* \in\argmax_{(v,y^v,z^v) \in \mathcal{V}^{k}} z^v$;
		\RETURN $v^*$, $\Vinfo$
	\end{algorithmic} \label{proc:DLSW}
\end{procedure}
\begin{align} \label{eq:NM}
\textrm{minimize } \norm{z} \:\: \textrm{subject to } f(x) \leq_C v+z, \:  x\in \X, z\in \R^p.
\end{align}	
\begin{procedure}[H]
	\caption{SelectVertex($V^k$,$\Vused$,$\Vinfo$) of Algorithm AUU}
	\begin{algorithmic}[1] 
		\STATE $ V_{\textnormal{info}}^{\textnormal{temp}} =\emptyset$;
		\FORALL {$v \in V^k \setminus \Vused$}	
		\IF{$(v,y^v, z^v) \notin \Vinfo$}
		\STATE Let \rev{$z^v$ be a solution to \eqref{eq:NM}};
		\ENDIF
		\STATE $V_{\textnormal{info}}^{\textnormal{temp}} \gets V_{\textnormal{info}}^{\textnormal{temp}} \cup \{(v,y^v,z^v)\}$;
		\ENDFOR
		\STATE $\Vinfo \gets V_{\textnormal{info}}^{\textnormal{temp}} $;
		\STATE Let $v^* \in\argmax_{(v,y^v,z^v) \in \mathcal{V}^{k}} z^v$;
		\RETURN $v^*$, $\Vinfo$
	\end{algorithmic}\label{proc:AUU}
\end{procedure}

\section{An example with different ordering cones}~\\ 
\rev{Consider Example 1 for $p\in \{2,3\}$ with the following ordering cones:\\
	$C_1=\cone \conv \{(1, 2)^\T, (2, 1)^\T\} \subsetneq \R^2_+, \\
	C_2=C_1^+=\cone \conv \{(2, -1)^\T, (-1, 2)^\T\} \supsetneq \R^2_+,\\
	C_3 = \cone \conv \{(-1, -1, 3)^\mathsf{T}, (2, 2, -1)^\mathsf{T}, (1, 0, 0)^\mathsf{T}, (0, -1, 2)^\mathsf{T}, (-1, 0, 2)^\mathsf{T}, (0, 1, 0)^\mathsf{T}\} \supsetneq \R^3_+,\\
	C_4=C_3^+=\cone \conv \{(4, 2, 2)^\mathsf{T}, (2, 4, 2)^\mathsf{T}, (4, 0, 2)^\mathsf{T}, (1, 0, 2)^\mathsf{T}, (0, 1, 2)^\mathsf{T}, (0, 4, 2)^\mathsf{T}\} \subsetneq \R^3_+$.\\
	Figure \ref{fig:vectorex} provides illustrations of the outer approximations of the upper image for $p=3$ with ordering cones $C_3$ and $C_4$.}


\begin{figure}[h]
	\centering
	\caption{\rev{The outer approximations of the upper image of Example 1 for $p=3$ with cones $C_3$ (left) and $C_4$ (right) obtained by Algorithm Adj, where the stopping condition is taken as $\epsilon = 0.01$.}}
	\begin{minipage}[b]{0.38\linewidth}
		\centering
		\includegraphics[width=\textwidth]{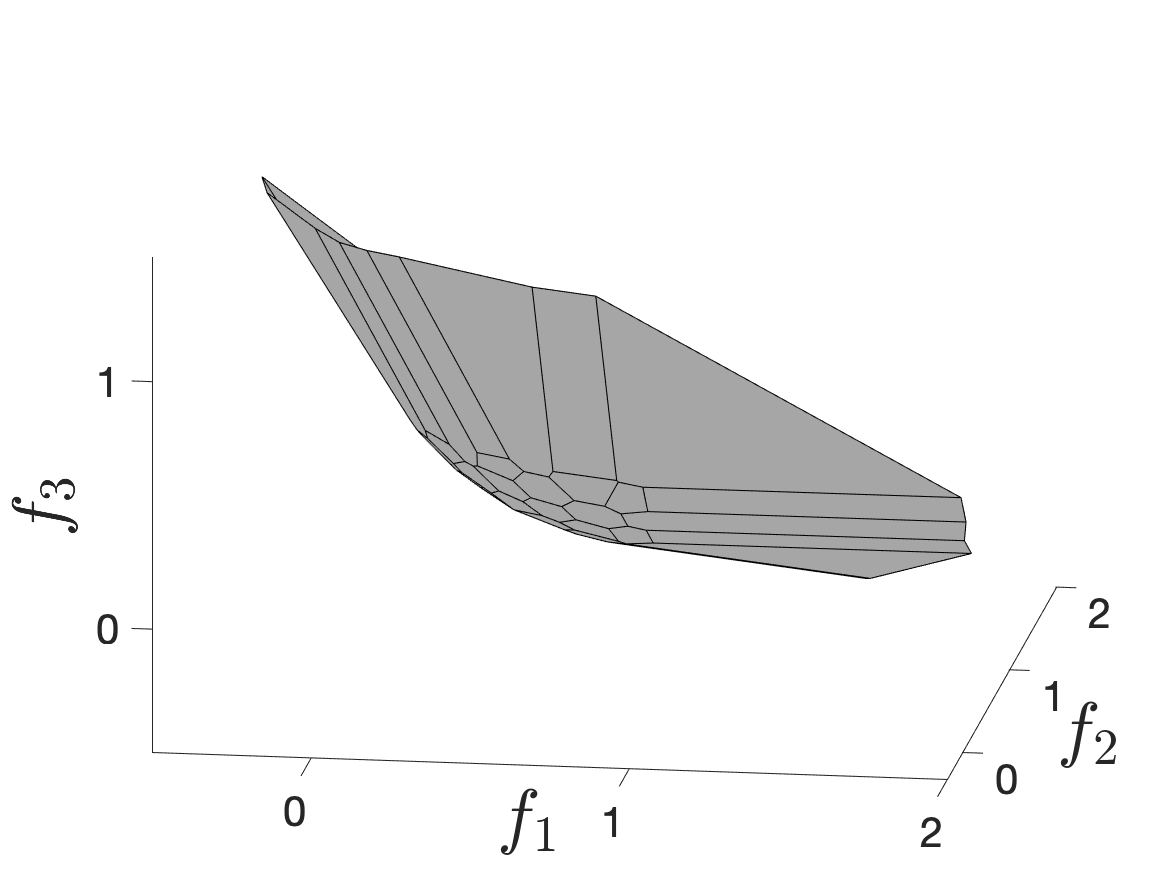}
	\end{minipage}
	\begin{minipage}[b]{0.38\linewidth}
		\centering
		\includegraphics[width=\textwidth]{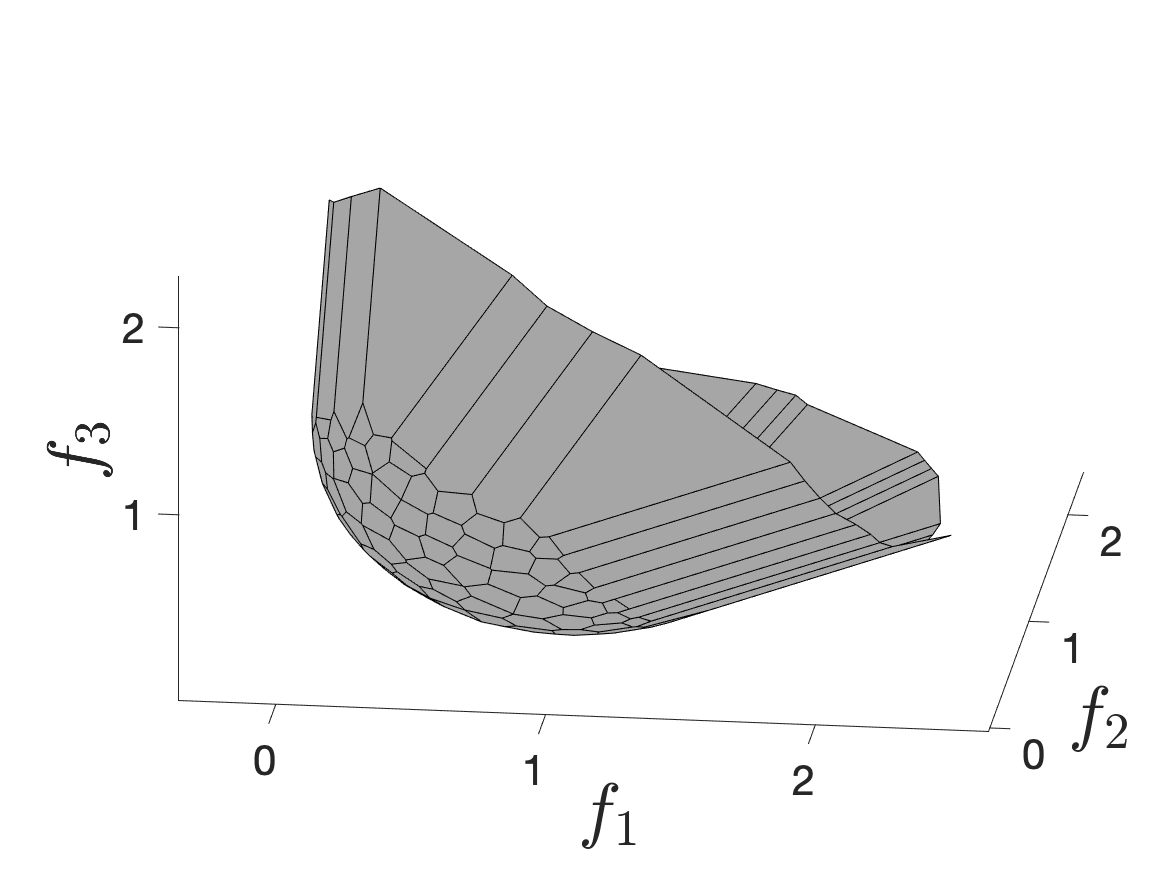}
	\end{minipage}
	\label{fig:vectorex}
\end{figure}

\rev{Since algorithms KTW and UB are designed to solve multiobjective optimization problems, they are not included in this set of experiments. Table \ref{tab:vectorex} shows the performances of AUU, DLSW, R, C, and Adj under different stopping criteria. Different from the \revv{main} set of experiments, the bounding set $\mathcal{Q}$ considered for HG calculations is computed exactly as in \cite[Section 6.1]{Simay}. Also, these experiments are conducted using a different computer with specifications: Intel(R) Core(TM) i7-8565U CPU @ 1.80GHz.} 


\begin{table}[H]
	\centering
	\caption{\rev{Computational results for Example 1 for $p=2$ with ordering cones $C_1$ and $C_2$, and for $p=3$ with ordering cones $C_3$ and $C_4$. The first set (fixed $\epsilon$) shows the results when the algorithms run until returning a finite weak $\epsilon$-solution for provided $\epsilon$ values; the second set (fixed T) shows the results when the algorithms run under runtime limit T, and the third set shows the results when the algorithms run until finding a solution set with predetermined cardinality (Card).}}
	\resizebox{\textwidth}{!}{ \begin{tabular}{c|c|cccc|cccc|cccc|cccc|}
			\cline{3-18}    \multicolumn{1}{r}{} &       & \multicolumn{4}{c|}{$p=2,\ C_1$}  & \multicolumn{4}{c|}{$p=2,\ C_2$}  & \multicolumn{4}{c|}{$p=3,\  C_3$}  & \multicolumn{4}{c|}{$p=3,\ C_4$} \\
			\cline{3-18}    \multicolumn{1}{r}{} &       & \multicolumn{4}{c|}{$\epsilon=0.0005$} & \multicolumn{4}{c|}{$\epsilon=0.0005$} & \multicolumn{4}{c|}{$\epsilon= 0.01$} & \multicolumn{4}{c|}{$\epsilon= 0.01$} \\
			\cline{2-18}          & Algorithm & VS    & SC    & T     & HG    & VS    & SC    & T     & HG    & VS    & SC    & T     & HG    & VS    & SC    & T     & HG \\
			\cline{2-18}    \multirow{5}[2]{*}{\begin{sideways} fixed $\epsilon$ \end{sideways}} & AUU   & 87    & \textbf{45}    & 25.93 & 0.0109 & 38    & \textbf{17}    & 10.72 & 0.00023 & 145   & \textbf{43}    & 46.24 & 0.00007 & 378   & \textbf{81}    & \textbf{105.58} & 0.2434 \\
			& DLSW  & 104   & \textbf{45}    & \textbf{23}    & 0.0054 & 37    & \textbf{17}    & 8.48  & \textbf{0.00001} & 310   & 75    & 65.96 & \textbf{0.00002} & 1109  & 148   & 196.43 & 0.0949 \\
			& R     & 0     & 129.2 & 24.46 & \textbf{0.0003} & 0     & 33.4  & 6.20  & 0.00004 & 0     & 106.6 & \textbf{23.32} & 0.00005 & 0     & 509   & 110.49 & 0.0225 \\
			& C     & 0     & 131   & 24.07 & \textbf{0.0003} & 0     & 33    & \textbf{5.69}  & 0.00006 & 0     & 114   & 28.25 & 0.00004 & 0     & 483   & 112.72 & 0.0271 \\
			& Adj   & 0     & 129   & 24.39 & \textbf{0.0003} & 0     & 33    & 6.04  & 0.00006 & 0     & 115   & 27.61 & 0.00003 & 0     & 601   & 137.42 & \textbf{0.0202} \\
			\cline{2-18}    \multicolumn{1}{r}{} &       & \multicolumn{4}{c|}{T=15}     & \multicolumn{4}{c|}{T=5}      & \multicolumn{4}{c|}{T=15}     & \multicolumn{4}{c|}{T=15} \\
			\cline{3-18}    \multicolumn{1}{r}{} &       & VS    & Card  & Err   & HG    & VS    & Card  & Err   & HG    & VS    & Card  & Err   & HG    & VS    & Card  & Err   & HG \\
			\cline{2-18}    \multirow{5}[2]{*}{\begin{sideways}fixed T\end{sideways}} & AUU   & 53    & \textbf{28}    & 0.0008 & 0.0116 & 19    & \textbf{11}    & 0.0008 & 0.00008 & 45    & \textbf{13}    & 0.0289 & 0.0012 & 166   & \textbf{36 }   & 0.0256 & - \\
			& DLSW  & 81    & 38    & 0.0008 & 0.0057 & 25    & 14    & 0.0008 & 0.00004 & 79    & 18    & 0.0165 & 0.0007 & 293   & 54    & \textbf{0.0205} & 0.0006 \\
			& R     & 0     & 56.2  & 0.0015 & 0.0013 & 0     & 16    & \textbf{0.0002} & \textbf{0.00001} & 0     & 31.2  & 0.0150 & \textbf{0.0003} & 0     & 126.2 & 0.0265 &\textbf{ 0.0001} \\
			& C     & 0     & 61    & 0.0015 & 0.0005 & 0     & 15    & \textbf{0.0002} & \textbf{0.00001} & 0     & 23    & 0.0162 & \textbf{0.0003} & 0     & 120   & 0.0226 & \textbf{0.0001} \\
			& Adj   & 0     & 62    & \textbf{0.0004} & \textbf{0.0003} & 0     & 16    & \textbf{0.0002} & \textbf{0.00001} & 0     & 35    & \textbf{0.0147} &\textbf{ 0.0003} & 0     & 129   & 0.0232 & 0.0003 \\
			\cline{2-18}    \multicolumn{1}{r}{} &       & \multicolumn{4}{c|}{Card=30}  & \multicolumn{4}{c|}{Card=10}  & \multicolumn{4}{c|}{Card=20}  & \multicolumn{4}{c|}{Card=40} \\
			\cline{3-18}    \multicolumn{1}{r}{} &       & VS    & T     & Err   & HG    & VS    & T     & Err   & HG    & VS    & T     & Err   & HG    & VS    & T     & Err   & HG \\
			\cline{2-18}    \multirow{5}[2]{*}{\begin{sideways}fixed Card \end{sideways}} & AUU   & 55    & 14.34 & \textbf{0.0004} & 0.0116 & 15    & 4.74  & 0.0008 & 0.00008 & 55    & 17.05 & 0.0437 & 0.0016 & 161   & 51.04 & 0.0262 & - \\
			& DLSW  & 67    & 14.43 & 0.0008 & 0.0060 & 19    & 5.05  & 0.0008 & 0.00005 & 90    & 26.81 & \textbf{0.0147} & 0.0008 & 203   & 45.06 & \textbf{0.0207} & 0.0006 \\
			& R     & 0     & \textbf{11.62} & 0.0015 & 0.0034 & 0     & 4.20  & \textbf{0.0002} & \textbf{0.00001} & 0     & \textbf{10.31} & 0.0598 & 0.0006 & 0     & 33.66 & 0.0609 & \textbf{0.0001} \\
			& C     & 0     & 13.63 & 0.0015 & 0.0005 & 0     & 5.03  &\textbf{ 0.0002} & \textbf{0.00001} & 0     & 11.47 & 0.0530 & \textbf{0.0004} & 0     & 46.07 & 0.0255 & \textbf{0.0001} \\
			& Adj   & 0     & 13.12 & \textbf{0.0004} & \textbf{0.0004} & 0     & \textbf{4.01}  & \textbf{0.0002} & \textbf{0.00001} & 0     & 11.78 & 0.0240 & 0.0005 & 0     & \textbf{33.08} & 0.0493 & - \\
			\cline{2-18}    \end{tabular}}%
	\label{tab:vectorex}%
\end{table}

\rev{From Table \ref{tab:vectorex}, we observe that if the algorithms run until they return a finite weak $\epsilon$-solution (fixed $\epsilon$), then similar to the observations from \revv{Section 6}, the proposed variants are faster than the algorithms from the literature for the ordering cones larger than the positive orthant ($C_2, C_3$). However, there is no clear conclusion for the ordering cones smaller than the positive orthant.}

\rev{Under fixed runtime (fixed T), we observe that AUU and DLSW return solution sets with smaller cardinality. The algorithms are comparable in terms of the Err that they return. Moreover, AUU consistently yields the worst HG value and R, C and Adj return slightly better HG values than DLSW.} 

\rev{Finally, when the algorithms run until they find a predetermined number of solutions (fixed Card), the proposed variants return smaller HG values compared to AUU and DLSW. However, T and Err values are comparable for all algorithms.}

\section{Results for different runtime limits for Example 3}

\begin{table}[!]
	\centering
	\caption{\rev{Computational results for Example 3 when the algorithms run under different runtime limits T$\in \{50, 100, 200\}$. }}
	\resizebox{\textwidth}{!}{
		\begin{tabular}{|c|l|cc|cc|cc|cc|cc|cc|}
			\cline{3-14}    \multicolumn{1}{c}{} & \multicolumn{1}{c|}{} & \multicolumn{6}{c|}{Example 3, $a=5$}                & \multicolumn{6}{c|}{Example 3, $a=7$}             \\ 
			\cline{3-14}    \multicolumn{1}{c}{} & \multicolumn{1}{c|}{} & \multicolumn{2}{c}{Err} & \multicolumn{2}{c}{HG} & \multicolumn{2}{c|}{Card} & \multicolumn{2}{c}{Err} & \multicolumn{2}{c}{HG} & \multicolumn{2}{c|}{Card} \\
			\cline{3-14}    \multicolumn{1}{c}{} & \multicolumn{1}{c|}{} & $\epsilon =0$     & $\epsilon =0.005$ & $\epsilon =0$    & $\epsilon =0.005$ & $\epsilon =0$     & \multicolumn{1}{c|}{$\epsilon =0.005$ } & $\epsilon =0$      & $\epsilon =0.005$  & $\epsilon =0$      & $\epsilon =0.005$ & $\epsilon =0$     & \multicolumn{1}{c|}{$\epsilon =0.005$} \\
			\hline
			\multicolumn{1}{|c|}{\multirow{7}[1]{*}{\begin{sideways}T=50\end{sideways}}} & AUU   & \textbf{0.061} & 0.061 & 2.5348 & 2.5348 & 38    & \multicolumn{1}{c|}{38} & 0.0635 & 0.0633 & 3.6097 & 3.4895 & {37}    & \multicolumn{1}{c|}{38} \\
			& DLSW  & 0.1428 & 0.1159 & 7.4127 & 7.2898 & {19}    & \multicolumn{1}{c|}{{20}} & \textbf{0.0577} & \textbf{0.0577} & 3.3893 & 3.3893 & 45    & \multicolumn{1}{c|}{45} \\
			& KTW   & 0.0633 & 0.0633 & 1.4557 & 1.586 & 43    & \multicolumn{1}{c|}{42} & 0.0689 & 0.0689 & 2.0931 & 2.0931 & 43    & \multicolumn{1}{c|}{43} \\ 
			& UB    & 0.415 & 0.0489 & 5.4621 & 0.6262 & 154   & 164 (44) & 0.2163 & 0.078 & 2.5031 & 0.9474 & 150   & 164 (44) \\ 
			& R     & 0.201 & 0.1922 & 2.7797 & 1.9551 & 156   & 161.8 (27.8) & 0.5199 & 0.1298 & 6.977 & 4.3267 & 158   & 161.6 ({26.2}) \\ 
			& C     & 0.0925 & \textbf{0.0472} & 1.1053 & 0.7466 & 133   & 162 (26) & 0.0659 & 0.0659 & 1.0019 & 1.0321 & 159   & 157 (34) \\ 
			& Adj   & 0.0775 & 0.0709 & \textbf{0.6087} & \textbf{0.6229} & 156   & 161 (26) & 0.0658 & 0.0658 & \textbf{0.7674} & \textbf{0.7543} & 159   & 163 (28) \\
			\hline
			\multicolumn{1}{|c|}{\multirow{7}[2]{*}{\begin{sideways}T=100\end{sideways}}} & AUU   & \textbf{0.0293} & \textbf{0.0319} & 1.5539 & 1.5924 & 59    & \multicolumn{1}{c|}{58} & 0.0394 & 0.0393 & 2.4166 & 2.4045 & {48}    & \multicolumn{1}{c|}{{49}} \\
			& DLSW  & 0.098 & 0.098 & 3.3152 & 3.3694 & {28}    & \multicolumn{1}{c|}{{27}} & \textbf{0.0333} & 0.0333 & 1.372 & 1.372 & 82    & \multicolumn{1}{c|}{82} \\ 
			& KTW   & 0.0355 & 0.0355 & 0.8512 & 0.8469 & 71    & \multicolumn{1}{c|}{72} & 0.0448 & 0.0448 & 2.5774 & 2.5774 & 68    & \multicolumn{1}{c|}{68} \\ 
			& UB    & 0.415 & 0.0489 & 5.4621 & 0.3433 & 248   & 249 (72) & 0.2163 & \textbf{0.0286} & 2.3386 & {-} & 240   & 261 (81) \\ 
			& R     & 0.1433 & 0.087 & 1.0544 & 1.4441 & 283   & 323.2 (114.4) & 0.2185 & 0.0732 & 2.042 & 1.3026 & 285   & 322.6 (106.6) \\ 
			& C     & 0.0448 & 0.0447 & 0.4774 & 0.3685 & 281   & 326 (111) & 0.0455 & 0.0659 & 0.51  & 0.6218 & 281   & 314 (103) \\
			& Adj   & 0.0422 & 0.0329 & \textbf{0.2926} &\textbf{ 0.29 } & 280   & 319 (112) & 0.0382 & 0.0345 & \textbf{0.4494} & \textbf{0.3984} & 287   & 323 (115) \\
			\hline
			\multicolumn{1}{|c|}{\multirow{7}[2]{*}{\begin{sideways}T=200\end{sideways}}} & AUU   & \textbf{0.016} & 0.0151 & 0.6718 & 0.6417 & 109   & \multicolumn{1}{c|}{114} & 0.0352 & {-} & 1.265 & {-} & 130   & -   \\
			& DLSW  & 0.0546 & 0.0546 & 2.4949 & 2.3092 & {40}    & \multicolumn{1}{c|}{{41}} & \textbf{0.0201} & 0.0201 & 0.7286 & 0.7329 & 154   & \multicolumn{1}{c|}{154} \\ 
			& KTW   & 0.0213 & 0.0213 & 0.5353 & 0.4808 & 122   & \multicolumn{1}{c|}{128} & 0.0225 & 0.0225 & 0.6514 & 0.667 & {129}   & \multicolumn{1}{c|}{{129}} \\
			& Ub    & 0.415 & \textbf{0.0108} & 5.462 & 0.1769 & 421   & 545 (275) & {-} &\textbf{ 0.0148 }& {-} & {-} & - & 427 (277) \\
			& R     & 0.1089 & 0.0405 & 1.6491 & 0.2699 & 464.6 & 650.2 (285.2) & 0.1281 & 0.0197 & 5.9576 & 0.6547 & 466.6 & 466.6 (297.8) \\ 
			& C     & 0.0448 & 0.0228 & 0.3711 & 0.2064 & 474   & 650 (279) & 0.0276 & 0.0659 & \textbf{0.5335} & 0.436 & 472   & 472 (288) \\
			& Adj   & 0.0405 & 0.0178 & \textbf{0.1778} & \textbf{0.177} & 463   & 641 (302) & 0.0374 & 0.0185 & 0.8885 &\textbf{ 0.394} & 469   & 469 (304) \\ 
			\hline
			\cline{3-14}    \multicolumn{1}{c}{} & \multicolumn{1}{c|}{} & \multicolumn{6}{c|}{Example 3, $a=10$}                & \multicolumn{6}{c|}{Example 3, $a=20$}  \\
			\cline{3-14}    \multicolumn{1}{c}{} & \multicolumn{1}{c|}{} & \multicolumn{2}{c}{Err} & \multicolumn{2}{c}{HG} & \multicolumn{2}{c|}{Card} & \multicolumn{2}{c}{Err} & \multicolumn{2}{c}{HG} & \multicolumn{2}{c|}{Card} \\
			\cline{3-14}    \multicolumn{1}{c}{} & \multicolumn{1}{c|}{} & $\epsilon =0$     & $\epsilon =0.005$ & $\epsilon =0$    & $\epsilon =0.005$ & $\epsilon =0$     & \multicolumn{1}{c|}{$\epsilon =0.005$ } & $\epsilon =0$      & $\epsilon =0.005$  & $\epsilon =0$      & $\epsilon =0.005$ & $\epsilon =0$     & \multicolumn{1}{c|}{$\epsilon =0.005$} \\
			\hline
			\multicolumn{1}{|c|}{\multirow{7}[1]{*}{\begin{sideways}T=50\end{sideways}}} & AUU   & \textbf{0.0603} & \textbf{0.0603} & 5.2053 & 5.2053 & 38    & \multicolumn{1}{c|}{38} & 0.0789 & 0.0789 & 12.7094 & 11.3389 & {33}    & \multicolumn{1}{c|}{33} \\
			& DLSW  & 0.1802 & 0.1802 & 8.6532 & 8.6532 & {23}    & \multicolumn{1}{c|}{23} & \textbf{0.0762} & \textbf{0.0741} & 8.8851 & 8.6309 & 45    & \multicolumn{1}{c|}{46} \\
			& KTW   & 0.064 & 0.0835 & 2.0931 & 3.243 & 48    & \multicolumn{1}{c|}{42} & 0.0945 & 0.0945 & 5.8671 & 5.8671 & 45    & \multicolumn{1}{c|}{45} \\
			& UB     & 0.5282 & 0.0996 & {-} & \textbf{1.1499} & 155   & 160 (38) & 0.3833 & 0.123 & 37.5565 & \textbf{3.9389} & 136   & 155 (34) \\
			& R     & 0.184 & 0.1225 & 5.8429 & 4.0538 & 157.8 & 160 (28.4) & 0.2774 & 0.2773 & 22.452 & 25.2676 & 157.4 & 161 ({22.8}) \\
			& C     & 0.1105 & 0.139 & 1.4924 & 1.5941 & 158   & 154 ({22}) & 0.1079 & 0.1303 & \textbf{3.6095} & 7.1808 & 155   & 164 (38) \\
			& Adj    & 0.1004 & 0.1004 & \textbf{1.3138} & 1.5617 & 159   & 162 (30) & 0.2199 & 0.1601 & 4.2782 & 5.4184 & 157   & 163 (23) \\
			\hline
			\multicolumn{1}{|c|}{\multirow{7}[2]{*}{\begin{sideways}T=100\end{sideways}}} & AUU   &\textbf{ 0.03}  & \textbf{0.0305} & 2.6334 & 2.6359 & 69    & \multicolumn{1}{c|}{68} & \textbf{0.0335} & \textbf{0.0339} & 5.9792 & 5.9851 & {67}    & \multicolumn{1}{c|}{{66}} \\
			& DLSW  & 0.1122 & 0.1122 & 7.8622 & 7.704 & {29}    & \multicolumn{1}{c|}{{30}} & 0.037 & 0.037 & 4.6299 & 4.7261 & 86    & \multicolumn{1}{c|}{85} \\
			& KTW   & 0.0635 & 0.0561 & 2.1896 & 1.9612 & 54    & \multicolumn{1}{c|}{58} & 0.0774 & 0.0774 & 3.2962 & 3.3143 & 77    & \multicolumn{1}{c|}{76} \\
			& UB    & 0.5282 & 0.0584 & 20.5524 & \textbf{0.5731} & 254   & 257 (77) & 0.3833 & 0.0584 & 37.5564 & 1.5682 & 268   & 292 (85) \\
			& R     & 0.3065 & 0.0929 & 10.6075 & 1.0173 & 282.8 & 316.8 (95.4) & 0.0985 & 0.1333 & 4.409 & 5.6698 & 312.8 & 312.8 (91) \\
			& C     & 0.0444 & 0.0422 & 0.9645 & 0.8517 & 267   & 312 (102) & 0.0472 & 0.0931 & \textbf{1.8988} & 1.9479 & 277   & 319 (126) \\
			& Adj    & 0.0477 & 0.0477 &\textbf{ 0.6971} & 0.5908 & 283   & 323 (110) & 0.1594 & 0.0903 & 2.0781 & \textbf{1.3881} & 285   & 315 (95) \\    
			\hline
			\multicolumn{1}{|c|}{\multirow{7}[2]{*}{\begin{sideways}T=200\end{sideways}}} & AUU    & \textbf{0.0221} & 0.0222 & 1.5587 & 1.5863 & 92    & \multicolumn{1}{c|}{90} & - & - & - & - & - & - \\
			& DLSW  & 0.0588 & 0.0603 & 4.8736 & 5.3437 & {58}    & \multicolumn{1}{c|}{{57}} & \textbf{0.0204} & 0.0204 & 2.0845 & 2.0861 & 158   & \multicolumn{1}{c|}{157} \\
			& KTW    & 0.0425 & 0.0425 & 1.2532 & 1.2717 & 116   & \multicolumn{1}{c|}{115} & 0.0774 & 0.0774 & 2.8347 & 2.8347 & {92}    & \multicolumn{1}{c|}{{92}} \\
			& Ub     & 0.5282 & \textbf{0.0154} & 20.5524 & 0.3031 & 423   & 541 (296) & 0.3832 & 0.0239 & 32.564 & \textbf{0.7241} & 423   & 522 (286) \\
			& R      & 0.0899 & 0.0839 & 1.2446 & 0.4661 & 463   & 628.8 (299.6) & 0.0778 & 0.0606 & 4.2709 & 1.6708 & 464   & 619 (320.4) \\
			& C      & 0.0366 & 0.0422 & 0.3878 & 0.3943 & 467   & 640 (296) & 0.0469 & \textbf{0.0168} & 2.3304 & 2.0154 & 466   & 632 (299) \\
			& Adj   & 0.0438 & 0.0319 & \textbf{0.3796} & \textbf{0.2962} & 468   & 623 (314) & 0.0903 & 0.0476 & \textbf{1.6607} & 1.0872 & 465   & 611 (330) \\
			\hline
			\hline
		\end{tabular}%
	}
	\label{tab:app}%
\end{table}


\end{document}